\documentclass[11pt,reqno]{amsart}

\usepackage{amsmath}
\usepackage{amsthm}
\usepackage{mathrsfs}
\usepackage{graphicx}
\usepackage{mathtools}
\usepackage{amsfonts}
\usepackage{amssymb}
\usepackage{hyperref}
\usepackage[margin=1 in]{geometry}
\usepackage{enumerate}
\usepackage[normalem]{ulem}
\usepackage{tikz}
\usepackage{tikz-cd}
\usepackage{xcolor}
\usetikzlibrary{matrix,arrows,positioning,automata}
  \usepackage{cancel}
\usepackage{dsfont}

  \usepackage{scalerel,stackengine}
\stackMath
\newcommand\reallywidehat[1]{%
\savestack{\tmpbox}{\stretchto{%
  \scaleto{%
    \scalerel*[\widthof{\ensuremath{#1}}]{\kern-.6pt\bigwedge\kern-.6pt}%
    {\rule[-\textheight/2]{1ex}{\textheight}}
  }{\textheight}%
}{0.5ex}}%
\stackon[1pt]{#1}{\tmpbox}%
}

\numberwithin{equation}{section}
\newcommand{\N}{\mathbb N}

\newcommand{\R}{\mathbb R}
\def\E{\mathbb E}
\def\P{\mathbb P}

\newcommand{\linf}{L^{\infty}}
\newcommand{\sF}{\mathcal{F}}
\newcommand{\bP}{\mathbb{P}}
\newcommand{\sL}{\mathcal{L}}
\newcommand{\bE}{\mathbb{E}}
\newcommand{\sG}{\mathcal{G}}

\newcommand{\sP}{\mathcal{P}}

\newcommand{\bx}{\bm{x}}
\newcommand{\bX}{\boldsymbol X}

\def\XXint#1#2#3{{\setbox0=\hbox{$#1{#2#3}{\int}$}
\vcenter{\hbox{$#2#3$}}\kern-.5\wd0}}

\newcommand{\T}{\mathbb{T}}

\numberwithin{equation}{section}
\newtheorem{thm}{Theorem}[section]
\newtheorem{lem}[thm]{Lemma}

\newtheorem{prop}[thm]{Proposition}
\newtheorem{assumption}[thm]{Assumption}
\theoremstyle{definition}

\newtheorem{rmk}[thm]{Remark}

\def\smallnegint{\mathop{\int\mkern-13mu
        \raise.5ex\hbox{${\scriptscriptstyle\diagup}$}}\nolimits}
\def\ds{\displaystyle}

\def\tr{\operatorname{tr}}

\def\bx{{\boldsymbol x}}
\def\by{{\boldsymbol y}}
\def\ssetminus{\,\raise.4ex\hbox{$\scriptstyle\setminus$}\,}

\newcommand{\be}{\begin{equation}}
\newcommand{\ee}{\end{equation}}

\newcommand{\bc}{\begin{case}}
\newcommand{\ec}{\end{cases}}
\newcommand{\bs}{\begin{split}}
\newcommand{\es}{\end{split}}

\newcommand{\bm}[1]{\boldsymbol #1}

\renewcommand{\d}{d}
\newcommand{\vs}{\vskip.075in}

\newcommand{\ov}{\overline}
\renewcommand{\tilde}{\widetilde}
\renewcommand{\hat}{\widehat}

\def \be {\begin{equation}}
\def \ee {\end{equation}}

\def \E {\mathbb{E}}
\def \P {\mathbb{P}}
\def \R {\mathbb{R}}

\renewcommand{\tilde}{\widetilde}

\newcommand{\pr}{\mathcal{P}}
\newcommand{\cC}{\mathcal{C}}

\newcommand{\cU}{\mathcal{U}}

\newcommand{\bbF}{\mathbb{F}}

\newcommand{\trip}[1]{{\left\vert\kern-0.25ex\left\vert\kern-0.25ex\left\vert #1 
    \right\vert\kern-0.25ex\right\vert\kern-0.25ex\right\vert}}
\newcommand{\bY}{\bm Y}
\newcommand{\cO}{\mathcal{O}}
\newcommand{\cT}{\mathcal{T}}
\newcommand{\bW}{\bm W}
\newcommand{\dtwo}{\mathbf{d}_2}
\newcommand{\dpee}{\mathbf{d}_p}
\newcommand{\bZ}{\mathbf{Z}}
\newcommand{\bd}{\mathbf{d}}
\colorlet{RED}{red}
\newcommand{\cR}{\mathcal{R}}
\newcommand{\intd}{\int_{\R^d}}
\newcommand{\cF}{\mathcal{F}}
\newcommand{\cG}{\mathcal{G}}
\newcommand{\pt}{\partial_t}
\newcommand{\cV}{\mathcal{V}}

\newcommand{\cQ}{\mathcal{Q}}
\newcommand{\dive}{\text{div}}

\newcommand{\cP}{\mathcal{P}}
\newcommand{\cK}{\mathcal{K}}

\thanks{P.~C. was partially supported by P.~E.~S's Air Force Office for Scientific Research grant FA9550-18-1-0494 and by the Agence Nationale de la Recherche (ANR), project ANR-22-CE40-0010 COSS. J.~J. was supported by the by the National Science Foundation  grant  DMS2302703. P.~E.~S. was partially  supported by the National Science Foundation grant DMS-1900599, the Office for Naval Research grant N000141712095 and the Air Force Office for Scientific Research grant FA9550-18-1-0494. }

\title[Sharp convergence rates for mean field control]{Sharp convergence rates for mean field control in the region of strong regularity}

\author[P. Cardaliaguet, J. Jackson, N. Mimikos-Stamatopoulos, and P. E. Souganidis]{Pierre Cardaliaguet, Joe Jackson, Nikiforos Mimikos-Stamatopoulos, \\
and Panagiotis E. Souganidis} 
\address{Universit\'e Paris-Dauphine, PSL Research University, Ceremade, 
Place du Mar\'echal de Lattre de Tassigny, 75775 Paris cedex 16 - France}
\email{cardaliaguet@ceremade.dauphine.fr }
\address{Department of Mathematics, University of Chicago, Chicago, Illinois 60637, USA}
\email{jsjackson@math.uchicago.edu}\address{Department of Mathematics, University of Chicago, Chicago, Illinois 60637, USA}
\email{nmimikos@math.uchicago.edu}
\address{Department of Mathematics, University of Chicago, Chicago, Illinois 60637, USA}
\email{souganidis@math.uchicago.edu}

\begin{document}

\begin{abstract}
We study the convergence problem for mean field control, also known as optimal control of McKean-Vlasov dynamics. We assume that the data is smooth but not convex, and thus the limiting value function $\cU:[0,T] \times \cP_2(\R^d) \to \R$ is Lipschitz, but may not be differentiable. In this setting, the first and last named authors recently identified an open and dense subset $\cO$ of $[0,T] \times \cP_2(\R^d)$ on which $\cU$ is $\cC^1$ and solves the relevant infinite-dimensional Hamilton-Jacobi equation in a classical sense. In the present paper, we use these regularity results, and some non-trivial extensions of them, to derive sharp rates of convergence. In particular, we show that the value functions for the $N$-particle control problems converge towards $\cU$ with a rate of $1/N$, uniformly on subsets of $\cO$ which are compact in the $p$-Wasserstein space for some $p > 2$. A similar result is also established at the level of the optimal feedback controls. The rate $1/N$ is the optimal rate in this setting even if $\cU$ is smooth, while, in general,  the optimal global rate of convergence is known to be slower than $1/N$. Thus our results show that the rate of convergence is faster inside of $\cO$ than it is outside. As a consequence of the convergence of the optimal feedbacks, we obtain a concentration inequality for optimal trajectories of the $N$-particle problem started from i.i.d. initial conditions.
\end{abstract}

\maketitle

\setcounter{tocdepth}{2}
\tableofcontents

\section{Introduction}

This paper is concerned with the convergence of certain high-dimensional stochastic control problems towards their mean field limits. To define these control problems, we fix throughout the paper a dimension $d \in \N$, a time horizon $T > 0$, and a filtered probability space $(\Omega, \bbF = (\sF_t)_{0 \leq t \leq T} , \bP)$ satisfying the usual conditions and hosting independent $d$-dimensional Brownian motions $W$ and $(W^i)_{i \in \N}$. 
\vs
The data consists of nice functions 
\begin{align*}
    L : \R^d \times \R^d \to \R, \quad \sF,\, \sG : \cP_2(\R^d) \to \R, 
\end{align*}
where $\cP_2(\R^d)$ is the Wasserstein space of Borel probability measures on $\R^d$ with finite second moment. Precise assumptions on $L$, $\sF$, and $\sG$ will be introduced in Subsection \ref{subsec:assump} below. 
\vs
The $N$-particle value function $\cV^N : [0,T] \times (\R^d)^N \to \R$ is defined by the formula 
\begin{align} \label{vndef}
    \cV^N(t_0,\bx_0) = \inf_{\bm \alpha \in \mathcal A}  \E\bigg[ \int_{t_0}^T \Big(\frac{1}{N} \sum_{i = 1}^N L(X_t^i, \alpha_t^i)  + \sF(m_{\bX_t}^N) \Big) dt + \sG(m_{\bX_T}^N) \bigg],
\end{align}
where $\mathcal A$ is the set  of  square-integrable, $\bbF$-adapted, $(\R^d)^N$-valued processes $\bm \alpha = (\alpha^1,...,\alpha^N)$ defined on $[t_0,T]$, and $\bX = (X^1,...,X^N)$ is the $(\R^d)^N$-valued state process which is determined from the control $\bm \alpha$ by the dynamics 
\begin{align*}
    dX_t^i = \alpha_t^i dt + \, \sqrt{2} dW_t^i,  \ \  t_0 \leq t \leq T, \quad X_{t_0}^i = x_0^i.
\end{align*}

We recall that under mild conditions on the data (see  Assumption \ref{assump:values} below), $\cV^N$ is the unique classical solution of the Hamilton-Jacobi-Bellman equation
\begin{align} \label{eq:hjbn} \tag{$\text{HJB}_N$}
    \begin{cases}     \displaystyle
    - \partial_t \cV^N - \sum_{i = 1}^N \Delta_{x^i} \cV^N + \frac{1}{N} \sum_{i = 1}^N H(x^i, N D_{x^i} \cV^N) = \sF(m_{\bx}^N) \ \ \text{in} \ \   [0,T] \times (\R^d)^N, \vspace{.2cm} \\
    \cV^N(T,\bx) = \sG(m_{\bx}^N) \ \ \text{for} \ \  \bx \in (\R^d)^N,
    \end{cases}
\end{align}
with the Hamiltonian  $H : \R^d \times \R^d \to \R$ given by $H(x,p) = \sup_{a \in \R^d} \big\{- a \cdot p - L(x,a) \big\}$. 
\vs

Next, we define the value function  $\cU : [0,T] \times \pr_2(\R^d) \to \R$ for the corresponding mean field problem by
\begin{align} \label{udef}
    \cU(t_0,m_0) &= \inf_{(m,\alpha)} \bigg\{ \int_{t_0}^T \bigg(\int_{\R^d} L(x,\alpha(t,x)) m_t(dx) + \sF(m_t)  \bigg) dt + \sG(m_T) \bigg\}, 
\end{align}
where the infimum is taken over all pairs $(m,\alpha)$ consisting of a curve $[t_0,T] \ni t \mapsto m_t \in \cP_2(\R^d)$ and a measurable map $\alpha : [t_0,T] \times \R^d \to \R$ such that 
\begin{equation}\label{takis1}
\begin{cases}
\int_{t_0}^T  \int_{\R^d} |\alpha(t,x)|^2 m_t(dx) dt < \infty, \ \ \text{and} \  \ 
\text{the Fokker-Planck equation}\\[1.5mm]

       \partial_t m = \Delta m - \text{div}(m \alpha) \  \text{ in } \  [t_0,T] \times \R^d, \quad m_{t_0} = m_0\\[1.2mm]
       
\text{    is satisfied in the sense of distributions.}
\end{cases}
\end{equation}
    
We recall that $\cU$ is expected to be the unique solution, in an appropriate viscosity sense, to the Hamilton-Jacobi equation

\begin{align} \label{eq:hjbinf} \tag{$\text{HJB}_{\infty}$}
    \begin{cases} \displaystyle
    - \partial_t \cU - \int_{\R^d} \tr(D_x D_m \cU) dm + \int_{\R^d} H(x,D_m \cU) dm = \sF(m) \ \ \text{in} \ \ [0,T] \times \pr_2(\R^d), \vspace{.2cm} \\
    \cU(T,m) = \sG(m) \ \text{in} \  \pr_2(\R^d);
    \end{cases}
\end{align}
see e.g. \cite{CossoGozziKharroubiPham, Soner2022, daudinseeger, bayraktar2023, djs2023} and the references therein for various approaches to the comparison principle for viscosity solutions of \eqref{eq:hjbinf}.

\subsection{Previous convergence results}

It is by now well understood that $\cV^N$ converges to $\cU$ in the sense that 
\begin{align} \label{valueconvintro}
    \cV^N(t,\bx) \approx \cU(t,m_{\bx}^N),  \quad \text{for $N$ large}.
\end{align}
This convergence was first established in \cite{Lackercontrol}, and later extended in \cite{DjetePossamaiTan} to allow the presence of a common noise. In another direction, \cite{GangboMayorgaSwiech} and \cite{mayorgaswiech} used PDE techniques to obtain similar results in a setting with purely common noise. We refer also to the works \cite{fornasier_lisini_orrieri_savaré_2019} and \cite{CAVAGNARI2022268} for a study of the deterministic case via $\Gamma$-convergence techniques, to \cite{daudinlimits} for an extension of the methods in \cite{Lackercontrol} to problems with state constraints, and to \cite{talbitouzizhang} for a similar convergence result in the setting of mean field optimal stopping.
All the works mentioned in the preceding paragraph  use techniques based on compactness, and so obtain only qualitative versions of the statement \eqref{valueconvintro}.
\vs
More recently, there have been a number of attempts to quantify the convergence of $\cV^N$ to $\cU$. On the one hand, when $\sF$ and $\sG$ are convex and sufficiently smooth, the value function $\cU$ is smooth, and a standard argument (see the introduction of \cite{cardaliaguet2023algebraic} for a more detailed explanation) shows that $|\cV^N(t,\bx) - \cU(t,m_{\bx}^N)| \leq C/N$. On the other hand, when $\cF$ and $\cG$ are not convex, then the value function $\cU$ may fail to be $\cC^1$ even if all the data is smooth (see \cite{BrianiCardaliaguet} for an example). In this setting the optimizers for the mean field control problem may not be unique, and obtaining quantitative convergence results is much more subtle.
\vs
The first general, that is, not requiring convexity or other special structure on $\sF$ and $\sG$, quantitative version of \eqref{valueconvintro} was obtained in \cite{cardaliaguet2023algebraic}, where the authors proved the estimate 
\begin{align} \label{cdjsrateintro}
    |\cV^N(t,\bx) - \cU(t,m_{\bx}^N) | \leq C\big(1 + \frac{1}{N} \sum_{i = 1}^N |x^i|^2 \big)N^{-\beta_d}, 
\end{align}
for  $C$ depending on all of the data and the exponent $\beta_d$ depending only on $d$. We refer also to \cite{CecchinFinite} for a thorough treatment of the finite state space setting and to \cite{bayraktarregimeswitching}, which obtains a sharp rate but under a special structural condition on the data. The more recent work \cite{ddj2023} attempts to identify the optimal rate of convergence, and shows in particular that the optimal rate depends on the smoothness of the data or, more precisely, the metric with respect to which the data is regular. Theorem 2.7 of \cite{ddj2023} shows that, if the data is periodic, that is, the state space $\R^d$ is replaced by the $d$-dimensional flat-torus $\T^d$, and sufficiently smooth (with the amount of smoothness required depending on the dimension $d$),
\eqref{cdjsrateintro} can be improved to 
\begin{align} \label{ddjrate}
    |\cV^N(t,\bx) - \cU(t,m_{\bx}^N)| \leq CN^{-1/2}, 
\end{align}
Example 2 in \cite{ddj2023}, meanwhile, shows that this rate cannot be improved even if all of the data is $\cC^{\infty}$. In summary, we now know that when the data is smooth and convex $\cU$ is smooth and the rate is $1/N$, but when the data is not convex, $\cU$ may fail to be smooth, and in this case the global rate is at best $1/\sqrt{N}$ even if all the data is very regular. 
\vs
There have also been some efforts to understand the convergence of the optimal trajectories and the optimal controls.
 For example, when $\cU$ is smooth one can follow the strategy initiated in \cite{CardaliaguetDelarueLasryLions} to show that optimal trajectories of the $N$-particle control problem converge (with a rate) to optimal trajectories of the mean field problem (see \cite{germain_pham_warin_2022} for details on this approach). In the non-convex regime, such questions are much more subtle since, as mentioned already, there may not be a unique optimal trajectory for the limiting problem. The recent work \cite{CardSoug2022} overcomes this issue by identifying an open and dense subset $\cO$ of $[0,T] \times \cP_2(\R^d)$ where the value function $\cU$ is $\cC^1$ and such that optimal trajectories started from initial conditions in $\cO$ are unique. In particular, it is shown in \cite{CardSoug2022} that a quantitative propagation of chaos can be established when starting from initial conditions in $\cO$.

\subsection{Our results}

The open and dense set $\cO$ identified in \cite{CardSoug2022} plays a central role in our results. In what follows, we will call the set $\cO$ the region of strong regularity by analogy with the terminology in \cite{FlemingSouganidis}, where the same language is used to describe a region where a certain first-order Hamilton-Jacobi equation has a classical solution which can be computed via the method of characteristics; more on this analogy shortly.
\vs
Our first result states that, locally inside $\cO$, the convergence rate $1/N$ can be achieved even if the data is not convex. More precisely, we show in Theorem \ref{thm.main} that, for each set $K \subset \cO$ which is compact in $\cP_p(\R^d)$ for some $p > 2$,  with $\cP_p(\R^d)$ denoting the $p$-Wasserstein space, there is a constant $C = C(K)$ such that, for  each $N \in \N$ and each $(t,\bx) \in [0,T] \times (\R^d)^N$ such that $(t,m_{\bx}^N) \in K$, 
\begin{align} \label{intromainthm}
       |\cV^N(t,\bx) - \cU(t,m_{\bx}^N)| \leq C/N. 
\end{align}
We refer to Remark \ref{rmk:lp} for a discussion of the role of compactness in $\cP_p$ with $p > 2$. Combined with Example 2 in \cite{ddj2023}, this shows that the optimal global convergence rate is different than the optimal rate of convergence within $\cO$. 
\vs
Example 2 in \cite{ddj2023} also explains why we claim that the set $\cO$ plays a similar role as the regions of strong regularity in \cite{FlemingSouganidis}. Indeed, it is explained there that when $\sF$ and $\sG$ depend on $m$ only through its mean $\ov{m}$, that is, $\sF(m) = f(\ov{m})$ and $\sG(m) = g(\ov{m})$, and $L = \frac{1}{2} |a|^2$ for simplicity, we have 
$$
\cU(t,m) = u(t,\ov{m}) \ \ \text{and} \ \   \cV^N(t,\bx) = v^N(t, \frac{1}{N} \sum_{i = 1}^N x^i), 
$$
where $u$ and $v^N$ are the solutions of the finite-dimensional PDEs
\begin{align*}
- \partial_t u + \frac{1}{2} |Du|^2 = f\ \ \text{in} \ \ [0,T)\times \R^d \ \ \text{and} \ \   u(T,x) = g,
\end{align*}
and 
\begin{align*}
    - \partial v^N - \frac{1}{N} \Delta v^N + \frac{1}{2} |Dv^N|^2 = f\ \ \text{in} \ \ [0,T)\times \R^d  \ \ \text{and} \ \  v^N(T,x) = g. 
\end{align*}
Thus, the convergence problem reduces to vanishing viscosity. Moreover, while the best global estimate for $|v^N - u|$ is $O(1/\sqrt{N})$, the expansion achieved in \cite{FlemingSouganidis} clearly shows that, locally uniformly on ``regions of strong regularity", $|v^N - u| = O(1/N)$ . Thus our Theorem \ref{thm.main} can be viewed as an infinite-dimensional (partial) analogue of the results in \cite{FlemingSouganidis}, with $\cO$ playing the role of the regions of strong regularity in \cite{FlemingSouganidis}.
\vs
Our second result shows that, when the data is smooth enough, a similar sharp rate of convergence can be obtained for the gradients. More precisely, Theorem \ref{thm.gradient} shows that, for each set $K \subset \cO$ which is compact in $\cP_p(\R^d)$ for some $p > 2$, there is a constant $C = C(K)$ such that, for each $N \in \N$ and each $(t,\bx) \in [0,T] \times (\R^d)^N$ such that $(t,m_{\bx}^N) \in K$, 
\begin{align} \label{introgradthm}
       |ND_{x^i} \cV^N(t,\bx) - D_m\cU(t,m_{\bx}^N, x^i)| \leq C/N. 
\end{align}
The main interest of \eqref{introgradthm} is that it demonstrates that optimal feedbacks for the $N$-particle problem converge toward the optimal feedback for the mean field problem; see Remark \ref{rmk:feedback} for more details.

Using the strong convergence of optimal feedbacks in \eqref{introgradthm}, we obtain in Proposition \ref{prop.concentration} a concentration inequality for the optimal trajectories of the $N$-particle problems when started from appropriate i.i.d. initial conditions. This result complements the quantitative propagation of chaos results in \cite{CardSoug2022}, and can also be compared to similar concentration results for mean field games obtained in \cite{Delarue2018FromTM} under the assumption that the master equation has a smooth solution. 
\vs
Finally, we mention that in order to obtain our main convergence results, we have to sharpen in various ways the regularity results in \cite{CardSoug2022}.  In particular we show in Theorem \ref{thm.c2} that, under appropriate regularity conditions, the second Wasserstein derivative $D_{mm} \cU$ exists and is continuous in the region of strong regularity.

\subsection{Strategy of the proof}

We explain here the strategy of proof for the estimate \eqref{valueconvintro}. To avoid unnecessary technicalities related to higher moments of the relevant probability measures, we  only discuss here the periodic case, that is,  $\R^d$ is replaced by the $d$-dimensional flat torus $\T^d$.
\vs
First, we note that, in view of the semi-concavity  but not semi-convexity estimates for $\mathcal U$,  the estimate 
\begin{align*}
    \cV^N(t,\bx) \leq \cU(t,m_{\bx}^N) + C/N
\end{align*}
in fact holds globally.
\vs
To complete the proof, we need to show that the symmetric inequality holds locally uniformly in $\cO$. For each fixed  $(t_0,m_0) \in \cO$  we  work with small tubes $\cT_r(t_0,m_0)$ of radius $r$ around the optimal trajectory for the mean field control problem started from $(t_0,m_0)$; see Subsection \ref{subsec:tubes} for the precise definition of $\cT_r(t_0,m_0)$. The key result is proved in Lemma \ref{lem.key}.  It says that, when $0 < r_1 \ll r_2 \ll 1$, the probability that the empirical measure associated to the optimally controlled state process started from $(t,m_{\bx}^N) \in \cT_{r_1}(t_0,m_0)$ exits the larger tube $\cT_{r_2}(t_0,m_0)$ decays algebraically in $N$. In the non-compact setting treated below, this algebraic decay is uniform only over the intersection of $\cT_{r_1}(t_0,m_0)$ with a large ball in $\cP_p$, but we ignore this subtlety in the introduction. 
\vs
More precisely, we show that, for each $(t,\bx)$ such that $(t,m_{\bx}^N) \in \cT_{r_1}(t_0,m_0)$, 
\begin{align} \label{keyintro}
    \bP\Big[s \mapsto (s,m^N_{\bX^{(t,\bx)}_s}) \text{ leaves } \cT_{r_2}(t_0,m_0) \Big] \leq CN^{-\gamma}, 
\end{align}
where $\bX^{(t,\bx)}$ denotes the optimal trajectory for the $N$-particle problem started from $(t_0,\bx_0)$. Lemma \ref{lem.key} relies crucially on the global convergence rate of $\cV^N$ to $\cU$ already established in \cite{cardaliaguet2023algebraic} and on an ``asymmetric" version of the propagation of chaos arguments in \cite{CardSoug2022}.
\vs
The next step of the argument is to use the fact that, in view of  the regularity of $\cU$ in $\cO$, $\cU^N(t,\bx) = \cU(t,m_{\bx}^N)$ nearly solves \eqref{eq:hjbn} on $\cT_{r}(t_0,m_0)$ for $r$ sufficiently small. In Lemma \ref{lem: ImprovedRateTubes}, we use this fact together with a verification argument to show that, for $r$ small and $(t,\bx)$ such that $(t,m_{\bx}^N) \in \cT_r(t_0,m_0)$, 
\begin{equation}\label{introimprove}
\begin{split} 
    &\cU(t,m_{\bx}^N) - \cV^N(t,\bx) \leq C/N \\
    &\qquad + \bP\Big[s \mapsto (s,m_{\bX^{(t,\bx)}_s}) \text{ leaves } \cT_{r}(t_0,m_0) \Big] \times \sup_{(s,m_{\by}^N) \in \cT_r(t_0,m_0)} \Big(\cU(s,\by^N) - \cV^N(s,m_{\by}^N) \Big). 
\end{split}
\end{equation}
Combining \eqref{introimprove} with \eqref{keyintro}, we show that the rate of convergence improves when the radius of the tube shrinks. More precisely, we establish, for $0 < r_1 \ll r_2 \ll 1$, an estimate of the form 
\begin{equation} \label{introiterate}
\begin{split}
   \sup_{(s,m_{\by}^N) \in \cT_{r_1}(t_0,m_0)}& \Big(\cU(s,m_{\by}^N) - \cV^N(s,\by) \Big) \leq C/N \\[1.2mm]
    &+ CN^{-\gamma} \times  \sup_{(s,m_{\by}^N) \in \cT_{r_2}(t_0,m_0)} \Big(\cU(s,m_{\by}^N) - \cV^N(s,\by) \Big),
\end{split}
\end{equation}
where, crucially, $\gamma$ is independent of $r_1$ and $r_2$. In particular, because $\gamma$ is uniform we can apply \eqref{introiterate} to a finite sequence of radii $r_2^{(1)} \gg r_1^{(1)} = r_2^{(2)} \gg r_1^{(2)} = r_2^{(3)} \gg ...\gg r_1^{(k)} = r_2^{(k)} $ to get that 
\begin{align*}
    \sup_{(s,m_{\by}^N) \in \cT_{r_1^{(k)}}(t_0,m_0)} \Big(\cU(s,m_{\by}^N) - \cV^N(s,\by) \Big) \leq C N^{- (1 \wedge k \gamma)}, 
\end{align*}
and so choosing $k$ large enough we conclude the existence of a small radius $r > 0$ such that the desired rate of convergence holds on the small tube $\cT_r(t_0,m_0)$. This is enough to establish the estimate uniformly over compact subsets of $\cO$ as desired.
\vs
The strategy of proof for the convergence of the gradients in \eqref{introgradthm} is similar, but more complicated because without the comparison principle it is harder to conclude an estimate analogous to \eqref{introimprove}. In addition, while the argument outlined above requires only minor refinements of the regularity results in \cite{CardSoug2022} to execute, the convergence of the gradients requires a new Lipschitz bound on $D_{mm} \cU$ (locally within $\cO$), which is obtained in Theorem \ref{thm: C^2Theorem}.

\subsection{Organization of the paper} In Section \ref{sec:mainresults}, we discuss notations and some preliminaries, and then state precisely our main results. Section \ref{sec:valuesconv} contains the proof of our first main convergence result, Theorem \ref{thm.main}. Section \ref{sec:gradientsconverge} contains the proof of the convergence of the gradients (Theorem \ref{thm.gradient}), and Section \ref{sec:concentration} contains the proof of the concentration inequality (Proposition \ref{prop.concentration}). Finally, in Section \ref{sec:regularity} we state and prove  a number of regularity results which are used in the earlier sections.

\section{Preliminaries and main results} \label{sec:mainresults}

\subsection{Basic notation}

We fix throughout the paper numbers $d \in \N$, $T > 0$. We work on a fixed filtered probability space $(\Omega, \bbF = (\sF_t)_{0 \leq t \leq T}, \bP)$, which hosts independent $d$-dimensional Brownian motions $(W^i)_{i \in \N}$. We use bold to write elements of $(\R^d)^N$ or processes taking values in $(\R^d)^N$, that is,  we write $\bx = (x^1,...,x^N) \in (\R^d)^N$ for a general element of $(\R^d)^N$. We denote by $\cP = \cP(\R^d)$ the space of probability measures on $\R^d$, and, for $q \in (1,\infty)$, we denote by $\cP_q = \cP_q(\R^d)$ the $q$-Wasserstein space, that is, 
the set of $m \in \cP(\R^d)$ such that $M_q(m) < \infty$, where 
\begin{align*}
    M_q(m) \coloneqq \int_{\R^d} |x|^q m(dx)
\end{align*}
is the $q^{\text{th}}-$moment of the measure $m$. 
We endow $\cP_q(\R^d)$ with the usual $q$-Wasserstein distance, which we denote by $\bd_q$.
\vs
We will make use of the calculus on the Wasserstein space as explained in \cite{CardaliaguetDelarueLasryLions} and \cite{CarmonaDelarue_book_I}.  In particular, for a sufficiently smooth $\phi: \cP_2(\R^d) \to \R$, we write $\frac{\delta \phi}{\delta m}(m,x) : \cP_2(\R^d) \times \R^d \to \R$ for the linear derivative of $\phi$, which is defined by the formula 
\begin{align*}
    \phi(m) - \phi(\ov{m}) = \int_0^1 \int_{\R^d} \frac{\delta \phi}{\delta m}(sm+(1-s)\ov{m},x) (m - \ov{m})(dx)ds, 
\end{align*}
together with the normalization convention 
\begin{align*}\label{cond: NormalizationCond}
    \int_{\R^d} \frac{\delta \phi}{\delta m}(m,x) m(dx) = 0.
\end{align*}
When $\frac{\delta \phi}{\delta m}$ exists and is differentiable in its second argument, we denote by $D_m \phi = D_x \frac{\delta \phi}{\delta m}$ the Wasserstein or so called Lions derivative 
\begin{align*}
    D_m \phi(m,x) = D_x \frac{\delta \phi}{\delta m}(m,x) : \cP_2(\R^d) \times \R^d \to \R^d.
\end{align*}
Higher derivatives are denoted similarly. 
\vs
For $k\in \N$, we denote by $\mathcal{C}^k(\mathcal{P}_2(\R^d))$, the space of functions $\phi:\mathcal{P}_2(\R^d)\rightarrow \R$, such that for all $i \in \{1,\cdots, k\}$ and multi-index $l\in \{0,1,..., i\}^d$ with $|l| + i \leq k$, the derivative $D^{(l)} D_m^i \phi$ exists, and is continuous and uniformly bounded. Finally, for $k,n\in \N$, we denote by $C^k(\R^d;\R^n)$, the functions $\psi:\R^d\rightarrow \R^n$ that are $k-$times continuously differentiable. When $n=1$ we will write $C^k$. Similar notation will be used for standard H\"older spaces, that is,  for $\alpha \in (0,1)$, $C^{k + \alpha}$ will be the space of functions in $C^k$ with bounded and $\alpha$-H\"older continuous derivatives up to order $k$.

\subsection{Assumptions} \label{subsec:assump}

The data for our problem consists of the three functions 
\begin{align*}
    L = L(x,a) : \R^d \times \R^d \to \R, \quad \sF = \sF(m) : \sP_2(\R^d) \to \R \ \ \text{and} \ \  \sG = \sG(m) : \sP_2(\R^d) \to \R.
\end{align*}
The Lagrangian $L$ determines the Hamiltonian $H = H(x,p) : \R^d \times \R^d \to \R$ by the formula 
\begin{align*}
    H(x,p) = \sup_{a \in \R^d} \Big\{- a \cdot p - L(x,a) \Big\}.
\end{align*}

For part of the paper, we will work with essentially the same assumptions as in \cite{CardSoug2022}, which we record here.

\begin{assumption} \label{assump:values}
The Hamiltonian  $H$ is  in $C^2(\R^d \times \R^d)$, and, for some $c, C > 0$ and all $(x,p) \in \R^d \times \R^d$, 
\begin{align}
    -C + c|p|^2 \leq H(x,p) \leq C + \frac{1}{c} |p|^2,
\end{align}
and 
\begin{align}
    |D_xH(x,p)| \leq C(1 + |p|)
\end{align}
Moreover, $H$ is locally strictly convex with respect to the last variable, that is, for each $R>0$, there exists $c_R>0$ such that, for all $(x,p) \in \R^d \times B_R$, 
\begin{align}
    D^2_{pp}H(x,p)\geq c_R I_d. 
\end{align}
Meanwhile, $\sF \in \cC^2(\cP_2(\R^d))$, and $\mathcal F$, $D_m\mathcal F$, $D^2_{ym}\mathcal F$ and $D^2_{mm}\mathcal F$ are uniformly bounded. Finally, $\mathcal G\in \cC^4(\cP_2(\R^d))$ with all derivatives up to order $4$ uniformly bounded.
\end{assumption}

In order to obtain more regularity and to study the convergence of the gradients of $\cV^N$, we require some additional smoothness, recorded here.

\begin{assumption} \label{assump:gradients}
The data $\cF,\cG$ and $H$ satisfy Assumption \ref{assump:values}. In addition $\cF \in \cC^3(\mathcal{P}_2(\R^d);\R)$ and, for $i=1,2,3$ and some $\delta \in (0,1)$,
\[\sup\limits_{m\in \mathcal{P}_2(\R^d)} \|\frac{\delta^{(i)} \cF}{\delta m^{(i)}}(m, \cdot)\|_{C^{2+\delta}((\R^d)^i;\R)}+\sup\limits_{m\in \mathcal{P}_2(\R^d)} \|\frac{\delta^{(i)} \cG}{\delta m^{(i)}}(m, \cdot)\|_{C^{2 +\delta}((\R^d)^i;\R)}<\infty .\]
\end{assumption}
\vs

\subsection{Preliminaries} \label{subsec:prelim}

In this section we recall some of the main results from the recent papers \cite{cardaliaguet2023algebraic} and \cite{CardSoug2022}. 
\vs
First, the main result of \cite{cardaliaguet2023algebraic} (Theorem 2.5 therein) shows that, under Assumption \ref{assump:values}, there exist constants $C > 0$ depending on the data, $\beta_d \in (0,1)$ depending only on $d$, such that, for all  $N \in \N$ and  $(t,\bx) \in [0,T] \times (\R^d)^N$,
\begin{align} \label{eq:cdjsrate}
    |\cV^N(t,\bx) - \cU(t,m_{\bx}^N)| \leq CN^{-\beta_d} \Big(1 + \frac{1}{N} \sum_{i = 1}^N |x^i|^2\Big).
\end{align}
In \cite{CardSoug2022}, meanwhile, the authors show that under Assumption \ref{assump:values}, there exists an open and dense  set $\cO \subset [0,T] \times \pr_2(\R^d)$ such that $\cU$ is $C^1$ and satisfies \eqref{eq:hjbinf} in a classical sense on $\cO$.
\vs
To define $\cO$ precisely, we first need some additional  notation and terminology. If $(m,\alpha)$ is an optimizer for the problem defining $\cU(t_0,m_0)$, then we call the curve $[t_0,T] \ni t \mapsto m_t \in \cP_2(\R^d)$ an optimal trajectory starting  from $(t_0,m_0)$. Assumption \ref{assump:values} is enough to guarantee the validity of a standard result from the theory of mean field control, namely, that, if $m$ is an optimal trajectory started from $(t_0,m_0)$, then there exists a unique function $u : [t_0,T] \times \R^d \to \R$, which is called  the multiplier associated to the optimal trajectory $m$, such that the pair $(u,m)$ solves the MFG system
\begin{align*}
\begin{cases} \ds
 -\pt u -\Delta u+H(x,Du)= \frac{\delta \cF}{\delta m}(m_t,x) \ \text{ in } \ (t_0,T)\times \R^d, \vspace{.1cm}\\ \ds
    \pt m-\Delta m-\dive(mD_pH(x,Du))=0 \ \text{ in } \ (t_0,T) \times \R^d, \vspace{.1cm}\\ \ds 
    m_{t_0} =m_0 \ \text{and} \  u(T,x)=\frac{\delta \cG}{\delta m} (m_T,x) \text{ in }\R^d.
\end{cases}
\end{align*}

The set $\cO$ is defined as the set of $(t_0,m_0) \in [0,T) \times \pr_2(\R^d)$ such that there is a unique optimal trajectory $m$ started from $(t_0,m_0)$, which is stable (see \cite[Definition 2.5]{CardSoug2022})  in the sense that, if $u$ is the corresponding multiplier, then $(z,w) = (0,0)$ is the only solution in the space
\begin{align*}
    C^{(1 + \delta)/2, 1 + \delta} \times C\big([t_0,T]; \big(C^{2 + \delta})'\big)
\end{align*}
to the linear system
\begin{align*}
    \begin{cases} \ds
        - \partial_t z - \Delta z + D_pH(x,Du) \cdot Dz = \frac{\delta \sF}{\delta m}(x,m(t))(\mu(t)) \ \text{in } \  (t_0,T) \times \R^d, \vspace{.1cm} \\ \ds
        \partial_t \mu - \Delta \mu - \text{div}(D_pH(x,Du) \mu) - \text{div}(D_{pp}H(x,Du) Dz m) = 0 \  \text{in } \ (t_0,T) \times \R^d, \vspace{.1cm} \\ \ds
        \mu(t_0) = 0 \ \text{and} \  z(T,x) = \frac{\delta \sG}{\delta m}(x,m(T))(\mu(T)) \ \text{in } \ \R^d.
    \end{cases}
\end{align*}
The set $\cO$ will play a crucial role in the results of the present paper, as well.

\subsection{Main results}

Our first main result shows that on compact with respect to $\cP_p$, for some $p > 2$, subsets of $\cO$, the rate \eqref{eq:cdjsrate} can be substantially sharpened.

\begin{thm} \label{thm.main}
    Let Assumption \ref{assump:values} hold, and assume that $p > 2$. Then, for each subset $K$ of $\cO$ which is compact in $\cP_p(\R^d)$, there is a constant $C = C(K)$ such that, for each $(t,\bx) \in [0,T] \times (\R^d)^N$ such that $(t,m_{\bx}^N) \in K$, 
    \begin{align} \label{eq.rate}
        |\cU(t,m_{\bx}^N) - \cV^N(t,\bx) | \leq C/N.
    \end{align}
\end{thm}

\begin{rmk}
As explained above, Example 2 in \cite{ddj2023} clearly shows that, even if all the data is $C^{\infty}$, we cannot expect a global convergence rate (of $V^N$ to $U$) better than $1/ \sqrt{N}$. Thus Theorem \ref{thm.main} shows that the convergence rate is generally different inside $\cO$ than it is outside of $\cO$.
\end{rmk}

\begin{rmk} \label{rmk:lp}
Theorem \ref{thm.main} is not a consequence only of the regularity of $\cU$ inside $\cO$, that is,  if we know only that $\cU$ is smooth on some arbitrary open set $\cO'$, it does not follow that the rate is $1/N$ inside $\cO'$ in the sense of \eqref{eq.rate}. Indeed, the invariance of $\cO$ under optimal trajectories, that is  the fact that optimal trajectories for the MFC problem which start in $\cO$ remain there, plays a crucial role throughout the proof of Theorem \ref{thm.main}.
\end{rmk}

\begin{rmk}
    The fact that the rate is uniform only over subsets of $\cO$ which are compact in $\cP_p$ for some  $p > 2$ is related to the fact that extra integrability is needed in order to obtain the convergence of empirical measures in the expected Wasserstein distance. More precisely, in order to obtain a key lemma (Lemma \ref{lem.key} below), we rely on the results of \cite{FournierGuillin} to bound the probability that an auxiliary particle system exits a small ``tube" around an optimal trajectory of the limiting problem, with the radius of the tube measured with respect to $\bd_2$.  For this it is necessary to  control  the $p^{\text{th}}-$moment for the initial condition of the particle system.
\end{rmk}

Our next result shows that a sharp rate of convergence can also be obtained for the gradients.

\begin{thm} \label{thm.gradient}
   Let Assumption \ref{assump:values} hold, and assume that  $p > 2$. Then, for each subset $K$ of $\cO$ which is compact in $\cP_p$, there is a constant $C = C(K)$ such that,  for each $i = 1,...,N$ and $(t,\bx) \in [0,T] \times (\R^d)^N$ such that $(t,m_{\bx}^N) \in K$, 
    \begin{align} \label{eq.rategradient}
        |D_m \cU(t,m_{\bx}^N,x^i) - ND_{x^i}\cV^N(t,\bx) | \leq C/N.
    \end{align}
\end{thm}

\begin{rmk} \label{rmk:feedback}
   The main interest of Theorem \ref{thm.gradient} is that it implies a convergence rate for the optimal feedbacks
    for the $N$-particle problem which are  given by
    \begin{align*}
        \alpha^{N,i}(t,\bx) = - D_p H\big(x^i, N D_{x^i} \cV^N(t,\bx)\big). 
    \end{align*}
That is,  for any initial condition $(t_0,\bx_0) \in [0,T] \times (\R^d)^N$, the optimizer $\bm \alpha$ for the minimization problem in \eqref{vndef} satisfies $\alpha^i_t = \alpha^{N,i}(t,\bX_t)$, $\bX$ denoting the optimal trajectory starting from $(t_0,\bx_0)$. Meanwhile, the optimal feedback for the mean field problem, at least for initial conditions $(t_0,m_0) \in \cO$, takes the form 
    \begin{align*}
        \alpha^{\text{MF}}(t,m,x) = - D_p H \big(x^i, D_m \cU(t,m,x) \big).
    \end{align*}
That is,  for any $(t_0,m_0) \in \cO$, the unique optimizer $\alpha$ for the minimization problem in \eqref{udef} satisfies $\alpha(t,x) = \alpha^{\text{MF}}(t,m_t,x)$, $m_t$ denoting the unique optimal trajectory started from $(t_0,m_0)$. Thus we can clearly infer from Theorem \ref{thm.gradient} that,  for each subset $K$ of $\cO$ which is compact in $\cP_p$, there exists $C = C(K)$ such that 
    \begin{align*}
       \sup_{ \{(t,\bx) : (t,m_{\bx}^N) \in K\}} |\alpha^{N,i}(t,\bx) - \alpha^{\text{MF}}(t,m_{\bx}^N,x^i)| \leq C/N. 
    \end{align*}
\end{rmk}

In \cite{CardSoug2022}, it is shown in Theorem 1.2 that the regularity of $\cU$ in $\cO$ implies a convergence for the optimal trajectories of $\cV^N$, in the spirit of propagation of chaos, at least for appropriate initial conditions. Using the convergence of the gradients obtained in Theorem \ref{thm.gradient}, we are able to supplement this result with the following concentration inequality.

\begin{prop} \label{prop.concentration}
    Let Assumption \ref{assump:gradients} hold, and fix $(t_0,m_0) \in \cO$ such that $m_0 \in \cP_p$ for some $p > 2$ and, in addition,  satisfies a quadratic transport-entropy inequality, that is,  there exists $\kappa > 0$ such that 
    \begin{align} \label{t2}
        \bd_2(m_0, \nu) \leq \kappa \cR( \nu | m_0) \  \text{if } \ \nu \ll m_0, 
    \end{align}
    where $\cR$ denotes the relative entropy, 
    \begin{align*}
        \cR(\nu | \mu) = \begin{cases}
            \int \frac{d \nu}{d \mu} \text{log} \frac{d \nu}{d\mu} d\mu & \  \text{if } \ \nu \ll \mu, \\[1.2mm]
            \infty & \text{otherwise}.
        \end{cases}
    \end{align*}
    For each $N$, denote by $\bX^N = (X^{N,1},...,X^{N,N})$ the solution of 
    \begin{align} \label{def.xnsymmetric}
        dX_t^{N,i} = - D_p H(X_t^{N,i}, N D_{x^i} V^N(t,\bX_t^N)) dt + \sqrt{2} dW_t^i, \quad X_{t_0}^i = \xi^i, 
    \end{align}
    where $(\xi^i)_{i \in \N}$ are i.i.d. with common law $m_0$. Then, for any $\epsilon, \eta > 0$, there exists a constant $C > 0$ such that,  for each $N \geq 1$,
    \begin{align} \label{eq.concentration}
        \P\Big[ \sup_{t_0 \leq t \leq T} \bd_2\big(m_{\bX_t^N}^N, m_t \big)  > \epsilon \Big] \leq C\exp\big(-N^{1 - \eta}\big),
    \end{align}
   where $(m,\alpha)$ denotes the unique optimizer for the problem defining $\cU(t_0,m_0)$. 
\end{prop}

\begin{rmk}
    Concentration results similar to Proposition \ref{prop.concentration} were obtained for mean field games in \cite{Delarue2018FromTM} under the assumption that the corresponding master equation has a smooth solution. The analogous condition in our setting would be that the value function $\cU$ is globally smooth, which we do not have because we do not assume any convexity on $\sF$ and $\sG$. 
\end{rmk}

\begin{rmk}
    The presence of the $\eta$ in \eqref{eq.concentration} comes from the fact that in Theorem \ref{thm.gradient} we have convergence of the optimal feedback strategies only on compact subsets of $\cP_p$, with $p > 2$. In particular, in the proof of Proposition \ref{prop.concentration} we must begin the argument by bounding from above the probability that the emprical measure of the optimal trajectory for the $N$-particle problem exits a large ball in $\cP_p$. This can be estimated from above by a multiple of a probability of the form 
    \begin{align}
        \bP\big[ \frac{1}{N} \sum_{i = 1}^N |\xi^i|^p > C\big]
    \end{align}
    for some large $C$, where $(\xi^i)_{i \in \N}$ are i.i.d. sub-Gaussian random variables. When $p \leq 2$, the random variables $|\xi^i|^p$ are i.i.d. random variables with sub-exponential tails, and so, for large enough $C$, we can indeed get $\bP\big[ \frac{1}{N} \sum_{i = 1}^N |\xi^i|^p > C\big] \leq C \exp(-cN)$. However, when $p > 2$, the random variables $|\xi^i|^p$ are only ``sub-Weybull" rather than sub-exponential, that is,  they have tails like $\bP[|\xi^i|^p > x] \approx \exp(-c x^{2/p})$, and this implies only the weaker estimate \begin{align}
        \bP\big[ \frac{1}{N} \sum_{i = 1}^N |\xi^i|^p > C\big] \leq C \exp(-cN^{2/p}).
    \end{align}
    The ability to choose $p$ arbitrarily close to $2$ leads to the arbitratrily small parameter $\eta$ appearing in \eqref{eq.concentration}.
\end{rmk}

A key step in proving Theorem \ref{thm.gradient} is showing that, under Assumption \ref{assump:gradients}, the $\cC^1$ regularity obtained in \cite{CardSoug2022} can be improved to $\cC^2$ regularity. This result is interesting in its own right, and so we state it here, alongside our main convergence results.

\begin{thm} \label{thm.c2}
    Under Assumption \ref{assump:gradients}, the derivative $D_{mm} \cU$ exists and is continuous in $\cO$. Moreover, for each $(t_0,m_0) \in \cO$, there exist constants $\delta, C > 0$ such that, for each $t$, $m_1,m_2$ with $|t - t_0| < \delta$, $\dtwo(m_0,m_i) < \delta$ and  $i = 1,2$, we have 
    \begin{align*}
        \sup_{x,y \in \R^d} |D_{mm} \cU(t,m_1,x,y) - D_{mm} \cU(t,m_2,x,y)| \leq C \bd_1(m_1,m_2).
    \end{align*}
\end{thm}

\section{The proof of Theorem \ref{thm.main}} \label{sec:valuesconv}
In this section we present the proof of Theorem \ref{thm.main}. For simplicity, we fix throughout the section a $p > 2$. 
\vs
It will be useful to note that one of the inequalities in Theorem \ref{thm.main} is relatively easy. Indeed, under the smoothness assumptions on $\sF$, $\sG$, it is not difficult to show the following.

\begin{prop} \label{prop.easyinequality}
    There is a constant $C$ such that, for all $N \in \N$ and for each $(t,\bx) \in [0,T] \times (\R^d)^N$, 
    \begin{align*}
        \cV^N(t,\bx) \leq \cU(t,m_{\bx}^N) + C/N\;.
    \end{align*}
\end{prop}

\begin{proof}
   We omit the proof, since it is almost identical the one of  the first inequality in Theorem 2.7 of \cite{ddj2023}. 
\end{proof}

\subsection{Tubes around optimal trajectories}
\label{subsec:tubes}

We now introduce some notation which will be useful in the proof of Theorem \ref{thm.main}. 
\vs

Given $(t_0,m_0) \in \cO$, we denote by $t \mapsto m_t^{(t_0,m_0)}$ the unique optimal trajectory for the limiting McKean-Vlasov control problem started from $(t_0,m_0)$. For simplicity, we extend $m^{(t_0,m_0)}$ by a constant to $[0,t_0]$, that is,  we define $m_t^{(t_0,m_0)} = m_0$ for $0 \leq t < t_0$. 
\vs
For $r > 0$, $\cT_{r}(t_0,m_0)$ is an open tube around the optimal trajectory started from $(t_0,m_0)$, that is, 
\begin{align*}
\cT_{r}(t_0,m_0) = \Big\{(t,m) : t \in (t_0 - r, T] \cap [0,T], \,\, \dtwo(m,m^{(t_0,m_0)}_t) < r \Big\}.
\end{align*}
When $(t_0,m_0)$ is understood from context, we write simply $\cT_{r}$. 
\vs
Because we are working in the whole space, we will often have to intersect the ``tubes" $\cT_{r}$ with bounded subsets of the Wasserstein space $\cP_p$. To facilitate this, we set, for $R > 0$, 
\begin{align*}
    \cQ_R = [0,T] \times B_R^p, 
\end{align*}
where $B_R^p$ is the ball of radius $R$ in $\cP_p$, centered at $\delta_0$, and  use the notation
\begin{align*}
    \cT_{r,R}(t_0,m_0) = \cT_{r}(t_0,m_0) \cap \cQ_R.
\end{align*}

We also need to project these sets down to finite-dimensional spaces. In particular, we will write $\cO^N = \{(t,\bx) \in [0,T] \times (\R^d)^N : (t,m_{\bx}^N) \in \cO\}$, and likewise we set
\begin{align*}
    \cT^N_{r,R}(t_0,m_0) = \Big\{(t,\bx) \in [0,T] \times (\R^d)^N : (t,m_{\bx}^N) \in \cT_{r,R}(t_0,m_0) \Big\}.
\end{align*}
Finally, $\bX^{N,t,\bx} = (X^{N,t,\bx,i})_{i = 1,...,N}$ is the optimal trajectory for the $N$-particle control problem started from $(t,\bx) \in [0,T] \times (\R^d)^N$, that is,  the solution of 
\begin{align}\label{eq: OptimalDynamicsForNParticle}
   \begin{cases} \displaystyle
       dX_s^{N,t,\bx,i} = - D_p H\big(X_s^{N,t,\bx,i}, N D_{x^i} \cV^N(s, \bX_s^{N,t,\bx})\big) ds + \sqrt{2} dW_s^i \ \ \text{in} \ \  [s,T] \vspace{.2cm}, \\
     X_t^{N,t,\bx,i} = x^i.
   \end{cases} 
\end{align}

To prove Theorem \ref{thm.main}, it will suffice to establish the following Proposition.

\begin{prop} \label{prop.tube}

Let $(t_0,m_0) \in \cO$ and $R > 0$. Then there exist $r, C > 0$ such that, for each $N \in \N$ and each $(t,\bx) \in \cT^N_{r,R}(t_0,m_0)$, 
\begin{align*}
    |\cV^N(t,\bx) - \cU(t,m_{\bx}^N)| \leq C/N.
\end{align*}
    
\end{prop}

Indeed, once Proposition \ref{prop.tube} is proved we can complete the proof of Theorem \ref{thm.main} as follows:
\begin{proof}[Proof of Theorem \ref{thm.main}]
    Fix $(t_0,m_0) \in \cO \cap \big( [0,T] \times \cP_p \big)$, and choose $R$ large enough that $(t_0,m_0) \in \cQ_R$. Thanks to Proposition \ref{prop.tube}, there exist $r, C > 0$ such that,   for all $(t,\bx)$ such that 
        $(t,m_{\bx}^N) \in \cT_{r,R}(t_0,m_0),$
    \begin{align} \label{rateomega}
        |\cV(t,\bx) - \cU(t,m_{\bx}^N)| \leq C/N .
    \end{align}
    In particular,  for each $(t_0,m_0) \in \cO \cap \big( [0,T] \times \cP_p \big)$, there exists a subset of $\cO$, which is open in $\cP_p$ and contains $(t_0,m_0)$, on which the convergence rate is $1/N$, in the sense of
    \eqref{rateomega}. This completes the proof.
\end{proof}

\subsection{The proof of Proposition \ref{prop.tube}}

We give here the proof of Proposition \ref{prop.tube}. We start by recording a few preliminary facts about the tubes $\cT_{r}(t_0,m_0)$ defined above.

\begin{lem} \label{lem.tubefacts}
    Given $(t_0,m_0) \in \cO$, there exists $r_0 > 0$ such that $\cT_{r_0}(t_0,m_0) \subset \cO$. Moreover, for any $r_2 > 0$, there exists $0 < r_1 < r_2$ such that optimal optimal trajectories from within $\cT_{r_1}(t_0,m_0)$ remain in $\cT_{r_2}(t_0,m_0)$, that is,  for each $(t,m) \in \cT_{r_1}(t_0,m_0)$ and each $t \leq s \leq T$, 
    \begin{align*}
        \bd_2\big(m_{s}^{(t,m)}, m^{(t_0,m_0)}_s \big) < r_2.
    \end{align*}
\end{lem}

\begin{proof}
    The first claim is a consequence of the fact that $\cO$ is open, as proved in \cite[Theorem 2.8]{CardSoug2022}. The second one can be inferred from \cite[Lemma 2.9]{CardSoug2022} together with the uniform 1/2-H\"older continuity of the optimal trajectories $s \mapsto m^{(t,m)}_s \in \cP_2$, which is in turn a consequence of the uniform boundedness of optimal controls (see Lemma 3.3 of \cite{cardaliaguet2023algebraic}).  
\end{proof}

The next task is to use the regularity of $\cU$ in $\cO$ to argue that the function
\begin{align} \label{undef}
    \cU^N(t,\bx) = \cU(t,m_{\bx}^N) : [0,T] \times (\R^d)^N \to \R
\end{align}
satisfies a PDE similar to \eqref{eq:hjbn}. First, notice that, if we assume that $\cU$ is $\cC^{1,2}$, then $cU^N$ is a classical solution in $\cO^N$ to 
\begin{align} \label{eq:error}
    - \partial_t \cU^N - \sum_{j = 1}^N \Delta_{x^j} \cU^N + \frac{1}{N} \sum_{j = 1}^N H(x^j, ND_{x^j} \cU^N) = \sF(m_{\bx}^N) + E^N(t,\bx)
\end{align}
where 
\begin{align} \label{def:error}
    E^N(t,\bx) = - \frac{1}{N^2} \sum_{j = 1}^N \tr\big(D_{mm} \cU(t,m_{\bx}^N, x^j,x^j) \big).
\end{align}
Under Assumption \ref{assump:values}, the main results of \cite{CardSoug2022} show that $\cU$ is $\cC^1$, but not necessarily $\cC^2$, so we cannot immediately conclude that $\cU^N$ satisfies \eqref{eq:error}. Nevertheless, Proposition \ref{prop.dmlip} shows that, uniformly in $x$ and locally uniformly in $(t,m)$ with respect to $\dtwo$,   $m \mapsto D_m U(t,m,x)$ is $\bd_1$-Lipschitz. 
In particular, for each $(t_0,m_0) \in \cO$, it is clear that we can choose $r$ small enough so that $m \mapsto D_m \cU(t,m,x)$ is Lipschitz with respect to $\bd_1$ uniformly over $\cT_{r}(t_0,m_0)$. It follows that
$\cU^N$ is $\cC^1$ on $\cT_{r}(t_0,m_0)$ with each partial derivative
\begin{align*}
    D_{x^i}\cU^N(t,\bx) = \frac{1}{N} D_m \cU(t,m_{\bx}^N, x^i)
\end{align*}
being uniformly Lipschitz in $\bx$. Moreover, arguing as in the proof of \cite[Proposition 5.1]{ddj2023}, one can show that the equation \eqref{eq:error} holds almost everywhere on $\cT^N_{r}(t_0,m_0)$, with an error term $E^N$ satisfying 
\begin{align} \label{errorbound}
    \|E^N\|_{\linf(\cT_{r})} \leq C/N.
\end{align}

We record this sequence of observations in the following lemma.

\begin{lem} \label{lem:uneqn}
For any $(t_0,m_0) \in \cO$, there exists $r > 0$ such that, for each $N \in \N$, the projection $\cU^N$ defined by \eqref{undef} lies in $\cC^1\big(\cT^N_{r}(t_0,m_0) \big)$, with spatial derivatives $D_{x^i} \cU^N$ being Lipschitz continuous in $\bx$, and such that \eqref{eq:error} holds almost everywhere, with the error function 
\begin{align}
    E^N \coloneqq  - \partial_t \cU^N - \sum_{j = 1}^N \Delta_{x^j} \cU^N + \frac{1}{N} \sum_{j = 1}^N H(x^j, ND_{x^j} \cU^N) - \sF(m_{\bx}^N)
\end{align}
satisfying the estimate \eqref{errorbound}.
\end{lem}

We now establish a sequence of technical lemmas. The first one explains how to estimate the probability that the empirical measure associated with an optimal trajectory for $\cV^N$ grows quickly in $\cP_p$.

\begin{lem} \label{lem:lpnorm}
    There is a constant $C_p$ depending on $p$ and the data with the following property: for each $N \in \N$ and  $(t_0,\bx_0) \in [0,T] \times (\R^d)^N$, 
    \begin{align}
        \bP\Big[ \sup_{t_0 \leq t \leq T} \dpee(m_{\bX_t^{N,t_0,\bx_0}}, m_{\bx_0}^N) \geq C_p \Big] \leq C_p/N. \label{cpdef}
    \end{align}
\end{lem}

\begin{proof}
   For simplicity, we set $\bX = (X^1,...,X^N) = \bX^{N,t_0,x_0}$, and let $\bm \alpha = (\alpha^1,...,\alpha^N)$ denote the optimal control started from $(t_0,x_0)$. Then, for each fixed $t \in [t_0,T]$, we have 
   \begin{align*}
       \bd_p^p(m_{\bX_t}^N, m_{\bx_0}^N) &\leq \frac{1}{N} \sum_{i = 1}^N |X_t^i - x_0^i|^p \leq \frac{1}{N} \sum_{i = 1}^N \big|\int_{t_0}^t \alpha_s^ids + W_t^i - W_{t_0}^i \big|^p .
   \end{align*}
   Using the fact that $\alpha^i$ is bounded independently of $N$ (see \cite[Lemma 1.7]{CardSoug2022}), we find easily that 
   \begin{align*}
       \sup_{t_0 \leq t \leq T} \bd_p^p(m_{\bX_t}^N, m_{\bx_0}^N) \leq C \Big( 1 + \frac{1}{N} \sum_{i = 1}^N \sup_{t_0 \leq t \leq T} |W_t^i - W_{t_0}^i|^p \Big).
   \end{align*}
   The result now follows easily from Chebyshev's inequality.
   
\end{proof}

The next lemma shows that the probability of the empirical measure exiting a tube decays algebraically, with a constant that is uniform over a smaller tube or, more precisely, uniform over the intersection of the smaller tube with a ball in $\cP_p$.

\begin{lem} \label{lem.key}
There exists a constant $\gamma_{p,d} \in (0,1)$, which depends only on $p$ and $d$, with the following property. Suppose that $(t_0,m_0) \in \cO \cap \big([0,T] \times \cP_p \big)$,  $0 < r_1 < r < r_2$ and $R_1,R_2 > 0$  are such that 
    \begin{enumerate}
    \item $R_2 - R_1 \geq C_p$, $C_p$ being the constant appearing in \eqref{cpdef}.
        \item $r_2$ is small enough so that $\ov{\cT_{r_2}(t_0,m_0)} \subset \cO$, and the conclusion of Lemma \ref{lem:uneqn} applies on $\cT_{r_2}(t_0,m_0)$.
        \item optimal trajectories started from inside $\cT_{r_1}(t_0,m_0)$ remain in $\cT_{r}(t_0,m_0)$, that is, 
        \begin{align*}
            \sup_{t \leq s \leq T} \dtwo\Big( m^{(t,m)}_s, m^{(t_0,m_0)}_s \Big) < r \ \  \text{for all } \ \  (t,m) \in \cT_{r_1}(t_0,m_0). 
        \end{align*}
    \end{enumerate}
Then there exists a constant $C$, which is independent of $N$, such that, for all  $N \in \N$ and $ (t,\bx) \in \cT_{r_1,R_1}^N$,
    \begin{align*}
        \bP[\tau^{N,t,\bx} < T] \leq C N^{- \gamma_{p,d} }, 
    \end{align*}
    where 
        $\tau^{N,t,\bx} = \inf \Big\{ s > t : (s, \bX_s^{N,t,\bx}) \in \big(\cT_{r_2,R_2}^N\big)^c \Big\} \wedge T.$
\end{lem}

\begin{proof}
    We fix $(t_0,m_0)$, $r_1 < r < r_2$ and $R_1<R_2$ as in the statement and $N \in \N$ and $(t,\bx) \in \cT^N_{r_1, R_1}$. For notational simplicity, we write $\bX = X^{N,t,\bx}$ and  $\tau = \tau^{N,t,\bx}$. At this point we have dropped the superscript for simplicity, and we note that in the rest of the argument generic constants $C$ may change from line to line, but they must not depend on $(t,\bx,N)$. 
\vs
    We are going to break the problem up by writing $\tau$ as $$ \tau = \tau_1 \wedge \tau_2,$$ where
\begin{equation}\label{takis10}
        \tau_1 = \inf \Big\{ s > t : (s, \bX_s^{N,t,\bx}) \in \big(\cT_{r_2}^N\big)^c \Big\}  \wedge T \ \ \text{and} \ \ 
        \tau_2 = \inf \Big\{ s > t : \d_p(m_{\bX_s}^N, \delta_0) > R_2 \Big\}  \wedge T.
\end{equation}
    In particular, we have 
    \begin{align}
        \bP[ \tau < T] \leq \bP[ \tau_1 < T] + \bP[ \tau_2 < T].
    \end{align}
   Moreover, since  Lemma \ref{lem:lpnorm} shows that 
    \begin{align*}
        \bP[\tau_2 < T] &= \bP\big[\sup_{t \leq s \leq T}  \dpee(m_{\bX_s}^N,\delta_0) > R_2 \big] \leq \bP\big[\sup_{t \leq s \leq T}  \dpee(m_{\bX_s}^N, m_{\bx}^N) > R_2 - R_1 \big] \leq C_p/N,
    \end{align*}
 it suffices to prove a corresponding estimate on $\bP[ \tau_1 < T]$. 
\vs
We are now going to employ an argument similar to the proof of Lemma 3.2 in \cite{CardSoug2022}, but with non-symmetric initial conditions. In particular, combining  Lemma \ref{lem:uneqn} and the It\^o-Krylov formula (see \cite{Krylov1980} Section 2.10), on $[t,\tau_1]$ we have the dynamics
\begin{equation*} 
\begin{split}
    d \cU^N(s,\bX_s) &= \Big(\partial_t \cU^N(s, \bX_s) + \sum_{j = 1}^N \Delta_{x^j} \cU^N(s,\bX_s) \\
    &- \sum_{j = 1}^N D_pH(X_s^i, N D_{x^j} \cV^N ) \cdot D_{x^j} \cU^N(x,\bX_s)\Big) ds  + \sqrt{2} \sum_{j = 1}^N D_{x^j} \cU^N(s,\bX_s) \cdot  dW_s^j  \\
    &= \Big(\frac{1}{N} \sum_{j = 1}^N H(X_s^i, N D_{x^j} \cU^N(s,\bX_s))  - \sum_{j = 1}^N D_pH(X_s^i, N D_{x^j} \cV^N ) \cdot D_{x^j} \cU^N(x,\bX_s)  \\
    & \qquad \qquad - E_s - \sF(m_{\bX_s}^N) \Big)ds + \sqrt{2} \sum_{j = 1}^N D_{x^j} \cU^N(s,\bX_s) dW_s^j \nonumber \\
    & \geq 
    \bigg( \frac{1}{N} \sum_{j = 1}^N \Big( - L\big(X_s^j,-D_pH(X_s^j, ND_{x^j} \cV^N(s,\bX_s)) \big) \nonumber \\ \nonumber &\qquad \qquad + C^{-1} \Big|D_pH(X_s^j, ND_{x^j} \cU^N(s,\bX_s)) - D_pH(X_s^j, ND_{x^j} \cV^N(s,\bX_s)) \Big|^2 \nonumber \\
    &\qquad \qquad - CN^{-1} - \sF(m_{\bX_s}^N) \bigg)ds + \sqrt{2} \sum_{j = 1}^N D_{x^j} \cU^N(s,\bX_s) dW_s^j.
\end{split}
\end{equation*}
In the last bound, we used the fact that the  strict convexity of $L$ in $a$ yields some $C>0$ such that,  for any $x \in \R^d$ and  $p,q \in \R^d$, 
\begin{equation}\label{convexityconsequence}
     H(x,p) - D_p H(x,q) \cdot p \geq - L(x, - D_p H(x,q)) + \frac{1}{C}|D_p H(x,p) - D_p H(x,q)|^2.
\end{equation}
Taking expectations in the inequality above for the dynamics of $\cU^N(s,\bX_s)$ 
and integrating from $t$ to $\tau$, we get 
\begin{equation} \label{driftest}
\begin{split}
    \E&\bigg[\int_t^{\tau} \frac{1}{CN} \sum_{j = 1}^N \Big|D_pH(X_s^j, ND_{x^j} \cU^N(s,\bX_s)) - D_pH(X_s^j, ND_{x^j} \cV^N(s,\bX_s)) \Big|^2 \bigg]\\
    &\leq \E[\cU^N(\tau, \bX_{\tau})] - \cU^N(t, \bx) + \E\bigg[\int_{t}^{\tau} \frac{1}{N} \sum_{j = 1}^N  L\big(X_s^j,-D_pH(X_s^j, ND_{x^j} \cV^N(s,\bX_s))\big)\\
    & \qquad \qquad + \sF(m_{\bX_s}^N) \bigg] + C/N \\
    &\leq C\big(1 + M_2(m_{\bx}^N)\big) N^{-\beta_d} + \E[\cV^N(\tau, \bX_{\tau})] - \cV^N(t, \bx) \\
    & + \E\bigg[\int_{t}^{\tau} \frac{1}{N} \sum_{j = 1}^N  L\big(X_s^j,-D_pH(X_s^j, ND_{x^j} \cV^N(s,\bX_s))\big) + \sF(m_{\bX_s}^N) \bigg] = CN^{-\beta_d}, 
\end{split}
\end{equation} 
with the last equality following from the fact that $\bX$ is the optimal trajectory for the $N$-particle problem, and where $\beta_d$ is the exponent appearing in \eqref{eq:cdjsrate}. We note that the dependence factor $(1 + M_2(m_{\bx}^N))$ was absorbed into the constant, keeping in mind that the constant is allowed to depend on the radius $r_2$ of the largest tube. 

Next, we are going to use \eqref{driftest} to compare the process $\bX$ to the process $\bY = (Y^i)_{i = 1,...,N}$ defined on a stochastic interval $[t, \sigma]$ by 
\begin{align*}
    \begin{cases}  \displaystyle
    d Y_s^i = - D_p H\big(Y_s^i, D_m \cU(s,m_{\bY_s}^N, Y_s^i)\big) ds  + \sqrt{2}dW_s^i \quad t \leq s \leq \sigma : \inf \big\{s > t : (s,\bY) \in \Big( \cT_{r_3}^N \Big)^c  \big\} \wedge T,
    \\
    Y_t^i = x^i.
    \end{cases}
\end{align*}
In particular, noting that we can rewrite the dynamics of $\bX$ as
\begin{align*}
    dX_s^i = \Big(- D_p H\big(X_s^i, D_m \cU(s,m_{\bX_s}^N, X_s^i)\big) + E_s^i \Big) ds + dW_s^i, 
\end{align*}
with 
\begin{align*}
    E_s^i =  D_p H\big(X_s^i, ND_{x^i} \cU^N(s,\bX_s)\big) - D_p H \big(X_s^i, ND_{x^i} \cV^N(s,\bX_s) \big),
\end{align*}
we easily get by computing the dynamics of $\frac{1}{N} \sum_{ j = 1}^N |X_s^i - Y_s^i|^2$ and applying Gronwall's inequality that 
\begin{align*}
    \E\Big[\sup_{t \leq s \leq \sigma \wedge \tau} \frac{1}{N} \sum_{i = 1}^N |X_s^i - Y_s^i|^2 ds \Big] \leq C \E\bigg[\int_{t}^{\sigma \wedge \tau} \frac{1}{N} \sum_{j = 1}^N |E_s^j|^2 ds \bigg] \leq CN^{-\beta_d}, 
\end{align*}
the last inequality coming from \eqref{driftest}. It follows that
\begin{align} \label{mxmybound}
    \E\Big[\sup_{t \leq s \leq \sigma \wedge \tau} \dtwo^2(m_{\bX_s}^N, m_{\bY_s}^N) \Big] \leq CN^{-\beta_d}. 
\end{align}
In particular, if we pick $r' \in (r, r_2)$, and set 
\begin{align*}
    \sigma' = \inf \big\{s > t ; (s, \bY_s) \in \big(\cT^N_{r'})^c \big\} \wedge T, 
\end{align*}
then,  using  Markov's inequality and \eqref{mxmybound} , we find 
\begin{align*}
\bP[\tau_1 < T] &\leq \bP[ \sigma' < T] + \bP[ \sigma' = T, \text{ and } \tau_1  < T] \\
&\leq \bP[ \sigma' < T] + \bP\Big[ \sup_{t \leq s \leq \sigma \wedge \tau} \dtwo\big(m_{\bX_s}^N, m_{\bY_s}^N \big) > r_2 - r' \Big] \\
&\leq \bP[ \sigma' < T] + CN^{-\beta_d/2}. 
\end{align*}
To complete the proof, we need only show that, for some $\gamma'_{d,p}>0$ depending only on $d$ and $p$, 
\begin{align*}
    \bP[ \sigma' < T] \leq C N^{ - \gamma'_{d,p}} .
\end{align*}
We would like to infer this from Lemma \ref{lem.nonsymmetric} below, which is a non-symmetric version of a standard ``propagation of chaos" result for interacting particle systems. We cannot, however,  apply Lemma \ref{lem.nonsymmetric} directly, since the map 
\begin{align*}
    (t,x,m) \mapsto - D_pH(x,D_m\cU(t,m,x))
\end{align*} is not globally defined, let alone globally Lipschitz. To overcome this, we simply extend it, choosing a map $b = b(t,x,m) : [0,T] \times \R^d \times\sP_2(\R^d) \to \R^d$ such that 
\begin{equation*}
\begin{cases}
    b(t,x,m) = - D_pH(x,D_m\cU(t,m,x)) \ \ \text{ for \ $(t,x,m) \in \cT_{r_2}(t_0,m_0)$},\\
    b \  \text{ is globally bounded and is Lipschitz in $(x,m)$ (with respect to the $\dtwo$-distance) uniformly in $t$}.
\end{cases}
\end{equation*}
This is possible thanks to the regularity of $\cU$ in $\cT_{r_2}$. Then, we  define the process $\bZ = (Z^i)_{i = 1,...,N}$ on the whole interval $[t,T]$ via 
\begin{align*}
    \begin{cases}  \displaystyle
    d Z_s^i = b(s,Z_s^i, m_{\bZ_s}^N) ds  + \sqrt{2} dW_s^i  \quad t \leq s \leq T,
    \\
    Z_t^i = x^i.
    \end{cases}
\end{align*}
It is easy to see that $\bZ = \bY$ on $[t,\sigma)$, and that the solution $Z$ of the corresponding McKean-Vlasov SDE 
\begin{align*}
    \begin{cases}  \displaystyle
    d Z_s = b(s,Z_s, \sL(Z_s)) ds  +  \sqrt{2} dW_s \quad t \leq s \leq T,
    \\
    Z_t \sim m_{\bx}^N
    \end{cases}
\end{align*}
is exactly the optimal trajectory for the mean field control problem, that is,  $\sL(Z_s) = m_s^{(t,m_{\bx}^N)}$. In particular, we can use Lemma \ref{lem.nonsymmetric} and Markov's inequality to conclude
\begin{align*}
    \bP[\sigma' < T] = \bP \Big[\sup_{ t \leq s \leq T} \dtwo(m_{\bZ_s}^N, m_s^{(t_0,m_0)}) > r' \Big] \leq \bP \Big[\sup_{ t \leq s \leq T} \dtwo(m_{\bZ_s}^N, \sL(Z_s)) > r' - r \Big] \leq CN^{- \gamma'_{d,p}},
\end{align*}
where in the first inequality we used the fact that $\dtwo(\sL(Z_s), m_s^{(t_0,m_0)}) = \dtwo(m_s^{(t,m^N_{\bx})}, m_s^{(t_0,m_0)}) < r$, by hypothesis. This completes the proof.

\end{proof}

The following lemma is a sort of ``non-symmetric" version of a classical propagation of chaos estimate for interacting particle systems, which was used in the proof of Lemma \ref{lem.key}.

\begin{lem} \label{lem.nonsymmetric}
Fix $b = b(t,x,m) : [0,T] \times \R^d \times \pr_2(\R^d) \to \R^d$, and suppose that $b$ is measurable, bounded, and Lipschitz in $(x,m)$ with respect to the $\dtwo$-distance, uniformly in $t$. Then there is a constant $C>0$ depending only on the $\linf$ and Lipschitz bounds on $b$, and a constant $\gamma_{d,p}'>0$ depending only on $d$ and $p$, such that, for any $N \in \N$ and  $(t_0,\bx_0) \in [0,T] \times (\R^d)^N$, 
\begin{align*}
    \E\bigg[ \sup_{t_0 \leq t \leq T} \dtwo\big(m_t^{(t_0,m_{\bx_0}^N)}, m^N_{\bX_t^{(t_0,\bx_0)}} \big) \bigg] \leq C M_p^{1/p}(m_{\bx_0}^N) N^{- \gamma_{d,p}'}, 
\end{align*}
where $\bX_t^{(t_0,\bx_0)} = (\bX_t^{(t_0,\bx_0),i})_{i = 1,...,N}$ denotes the solution to the SDE 
\begin{align} \label{xdefnonsymm}
    dX_t^i = b(t, X_t^i, m_{\bX_t}) dt + \sqrt{2} dW_t^i \ \ \text{and} \ \ X_{t_0}^i = x_0^i, 
\end{align}
and $m_t^{(t_0,m_{\bx_0}^N)} = \sL(Y_t)$, where $Y$ solves 
\begin{align*}
    dY_t = b(t, Y_t, \sL(Y_t)) dt + \sqrt{2} dW_t \ \ \text{and} \ \  Y_{t_0} \sim m_{\bx}^N. 
\end{align*}
\end{lem}

\begin{proof}
    We fix $(t_0,\bx_0) \in [0,T] \times (\R^d)^N$. Thanks to the Lipschitz continuity of $b$, it is straightforward to check that the map $\Phi = \Phi(\bx) : (\R^d)^N \to \R$ given by 
    \begin{align*}
        \Phi(\bx) =  \E\bigg[ \sup_{t_0 \leq t \leq T} \dtwo\big(m_t^{(t_0,m_{\bx_0}^N)}, m^N_{\bX_t^{(t_0,\bx)}} \big) \bigg] 
    \end{align*}
    is Lipschitz continuous with respect to $\dtwo$ with a constant $C$ depending only on $L$, that is, 
    \begin{align*}
       | \Phi(\bx) - \Phi(\by) | \leq C \dtwo(m_{\bx}^N, m_{\by}^N). 
    \end{align*}
As a consequence,    
    we can find a function $\tilde{\Phi} : \pr_2(\R^d) \to \R$ with the same $\dtwo$-Lipschitz constant and such that $\tilde{\Phi}(m_{\bx}^N) = \Phi(\bx)$ for $\bx \in (\R^d)^N$. 
    \vs
    Now define the ``lift" $\hat{\Phi} : \pr_2(\R^d) \to \R$ by
    \begin{align*}
        \hat{\Phi}(m) = \int_{(\R^d)^N} \Phi(\by) d m^{\otimes N}(d\by). 
    \end{align*}
It follows  from Theorem 1 of \cite{FournierGuillin} that, for some $\gamma'_{p,d}>0$ depending explicitly on $p$ and $d$.
    \begin{align} \label{phihatphitilde}
        |\Phi(\bx^N) - \hat{\Phi}(m_{\bx}^N)| &= |\tilde{\Phi}(m_{\bx}^N) - \int_{(\R^d)^N} \tilde{\Phi}(m_{\by}^N) d(m_{\bx}^N)^{\otimes N}(d\by)| \nonumber  \\ &
        \leq C \int_{(\R^d)^N} \dtwo\big(m_{\by}^N, m_{\bx}^N \big) m_{\bx}^N(d\by) \leq C M^{1/p}_p(m_{\bx}^N) N^{- \gamma'_{p,d}}.
    \end{align}
We can also write 
    \begin{align*}
        \hat{\Phi}(m_{\bx_0}) = \E\bigg[ \sup_{t_0 \leq t \leq T} \dtwo\big(m_t^{(t_0,m_{\bx_0}^N)}, m_{\bY_t}^N \big) \bigg],
    \end{align*}
    where $\bY = (Y^1,..,Y^N)$ is defined using the same dynamics as in \eqref{xdefnonsymm} but with initial conditions $Y^i_{t_0} = \xi^i$, the $(\xi^i)_{i = 1,...,N}$ being i.i.d. with common law $m_{\bx_0}^N$. 
\vs    
It follows,  from a standard ``asynchronous coupling argument" (see, for example  the Proof of Theorem 1.10 in \cite{carmonabsdes}), that 
    \begin{align*}
        \hat{\Phi}(m_{\bx_0}) \leq C(1 + M_q(m_{\bx_0}^N)) N^{ - \gamma_{p,d}'},
    \end{align*}
    which completes the proof.
\end{proof}

The next lemma shows how the estimate from Lemma \ref{lem.key} can be used to improve the rate of convergence on small tubes.

\begin{lem}\label{lem: ImprovedRateTubes}
   Fix $(t_0,m_0)\in \cO$ and assume that  $r>0$ be small enough so that the conclusion of Lemma \ref{lem:uneqn} holds. Then, for $R > 0$, there exists a constant $C>0$ independent of $N$ such that, for all $(t,\bx)\in \cT_{r,R}^N$, 
    \[\cU^N(t,\bx)-\cV^N(t,\bx)\leq \frac{C}{N}+\sup\limits_{(s,\by)\in \cT_{r,R}^N}\Big(\cU^N(s,\by)-\cV^N(s,\by) \Big)\mathbb{P}[\tau^{N,t,\bx}<T],\] 
    where 
    \begin{align*}
        \tau^{N,t,\bx} \coloneqq \inf \{ s > t : \bX_s^{N,t,\bx} \in (\cT_{r,R}^N)^c \} \wedge T.
    \end{align*}
\end{lem}

\begin{proof}
    Recall that $U^N$ satisfies  \eqref{eq:error} on $\cT_{r}^N = \cT_{r}^N(t_0,m_0)$. 
Therefore, using a standard verification argument, we find, for any $(t,\bx)\in \cT_{r}^N$,  the formula
    \begin{equation} \label{unlocal}
    \begin{split}
     \cU^N(t,\bx)=\inf\limits_{\bm \alpha = (\alpha^1,...,\alpha^N)}\mathbb{E}\bigg[\int_t^{\tau} \Big(\frac{1}{N}\sum\limits_{i=1}^N L(X_s^i,\alpha_s^i)+\cF(m_{\bX_s}^N)\\
     +E^N(s,\bX_s) \Big) ds+\cU^N(\tau,\bX_{\tau}) \bigg],
    \end{split}
    \end{equation}
    subject to the dynamics
    \begin{align*}
        dX_s^i = \alpha_s^i ds +\sqrt{2}dW_s^i \quad t \leq s \leq \tau = \inf \big\{ s > t : X_s^i \in (\cT_{r}^N)^c \big\} \ \ \text{and} \ \  
        X_t^i = x^i.
    \end{align*}
    We note that both $X$ and $\tau$ are impacted by the choice of $\bm \alpha$. 
    \vs 
    Similarly, we have
\begin{align} \label{vnlocal}
\cV^N(t,\bx)=\inf \limits_{\bm \alpha}\mathbb{E}\bigg[\int_t^{\tau}\frac{1}{N}\sum\limits_{i=1}^NL(X_s^i,\alpha_s^i)+\cF(m_{\bX_s}^N) ds+ \cV^N(\tau,\bX_{\tau}) \bigg],
\end{align}
subject to the same dynamics. 
\vs
By dynamic programming, the optimal control for the optimization problem in \eqref{vnlocal} is the feedback $\alpha_s^i = - D_p H(X_s^i,  ND_{x^i} \cV^N(s,\bX_s))$, and when this control is played $\bX = \bX^{N,t,\bx}$ and  $\tau = \tau^{N,t,\bx}$. 
\vs
Thus testing this control in the optimization problem \eqref{unlocal}, we find

    \begin{align*} 
    \cU^N(t,\bx)-\cV^N(t,\bx) &\leq \mathbb{E}\Big[\int_{t}^{\tau^{N,t,\bx}}E_N(s,\bX_s)ds\Big]+\mathbb{E}\Big[\cU^N(\tau^{N,t,\bx},\bX_{\tau^{N,t,\bx}})-\cV^N(\tau^{N,t,\bx},\bX_{\tau^{N,t,\bx}}) \Big] 
    \\
    & \leq \frac{C}{N}+\mathbb{E}\Big[\Big(\cU^N(\tau^{N,t,\bx},X_{\tau})-\cV^N(\tau^{N,t,\bx},X_{\tau^{N,t,\bx}})\Big) \mathbf{1}_{\tau^{N,t,\bx} <T} 
    \\
    &\qquad  + \Big(\cU^N(\tau^{N,t,\bx},X_{\tau^{N,t,\bx}})-\cV^N(\tau^{N,t,\bx},X_{\tau^{N,t,\bx}})\Big) \mathbf{1}_{\tau^{N,t,\bx} =T}\Big]
    \\ 
    &\leq \frac{C}{N}+\sup\limits_{(t,\bx)\in \cT_{r}^N}\Big(\cU^N-\cV^N \Big)\mathbb{P}\Big[\tau^{N,t,\bx}<T \Big],
    \end{align*}
    where we used in the first inequality  the fact that $|E^N| \leq C/N$ and in the last inequality the fact that $U^N(T,\bx)=\cG(m_{\bx}^N)=V^N(T,\bx)$.

\end{proof}

We  now proceed with the proof of Proposition \ref{prop.tube}.

\begin{proof}[Proof of Proposition \ref{prop.tube}]
    Let $\gamma_{d,p}$ be the exponent appearing in Lemma \ref{lem.key}, and choose $M \in \N$ with $M -1 \geq 1 / \gamma_{d,p}$,  $R_M > 0$ large enough so  that $m_0 \in B_{R_M}^p$, and, for $m = 2,...,M$, define $R_{m-1}$ inductively by $R_{m-1} = R_m + C_p$, $C_p$ being the constant appearing in \eqref{cpdef} or, in other words, $C_m = C_M + (M-m) C_p$ for $m = 1,...,M$. The important point is that we have
    \begin{align*}
        m_0 \in B_{R_M}^p \subset B_{R_{M-1}}^p \subset ... \subset B_{R_1}^p \  \text{ and } \ R_{m} - R_{m - 1} = C_P.
    \end{align*}
    Next, choose
     $r^{(1)}_2 > 0$ small enough so that the conclusion of Lemma \ref{lem:uneqn} holds, and then construct, for $i = 1,2$ and $m = 1,...,M,$,  $r^{(m)}_{i}$, 
in such a way that 
   \[  \begin{cases}
         r^{(m+1)}_2 = r^{(m)}_1 < r^{(m)}_2 \ \ \text{for} \ \  m = 1,...,M-1, \text{and} \\[1.2mm]
         \text{there exists $r^{(m)}$ such that $r^{(m)}_1 < r^{(m)} < r^{(m)}_2$ satisfy the hypotheses of Lemma \ref{lem.key}. }
     \end{cases}\]
     Lemma \ref{lem.tubefacts} makes it clear that we can choose $r^{(m)}_i$ satisfying these conditions. To be clear, the important point is that we have
     \begin{align}
     0 < r_1^{(M)} < r_2^{(M)} = r_1^{(M-1)} < ... < r_2^{(2)} = r_1^{(1)} < r_2^{(1)}, 
    \end{align}
    and that, for $m = 1,...,M$, there is a $r^{(m)}$ such that the triple $r^{(m)}_1 < r^{(m)} < r^{(m)}_{2}$ satisfies the hypotheses of Lemma \ref{lem.key}.
 \vs   
    Combining Lemma~\ref{lem: ImprovedRateTubes} and Lemma~\ref{lem.key}, we find that, for each $m = 2,...,M$,
    \begin{equation} \label{firststep}
    \begin{split}
        \sup_{(t,\bx) \in \cT^N_{r_1^{(m)},R_m}} \Big( \cU^N(t,\bx)-\cV^N(t,\bx) \Big) &\leq \frac{C}{N}+\sup\limits_{(t,\bx)\in \cT_{r_2^{(m)}, R_{m-1}}^N}\Big(\cU^N-\cV^N \Big)N^{- \gamma_{d,p}}\\
        &= \frac{C}{N}+\sup\limits_{(t,\bx) \in \cT_{r_1^{(m-1)}, R_{m - 1}}^N}\Big(\cU^N-\cV^N \Big)N^{- \gamma_{d,p}}
    \end{split}
    \end{equation}
    Then \eqref{firststep} and  the global estimate \eqref{eq:cdjsrate} easily yield, through an inductive argument starting with $m = 1$,  that,     for each $m = 1,...,M$,
    \begin{align*}
        \sup_{(t,\bx) \in \cT^N_{r_1^{(m)},R_m}} \Big( \cU^N(t,\bx)-\cV^N(t,\bx) \Big) \leq C/N^{1 \wedge \big(\beta_d + (m-1) \gamma_{d,p}\big)}
    \end{align*}
    Since $M -1 > 1/ \gamma_{d,p} \geq 1$ by design, we see in particular that 
    \begin{align*}
        \sup_{(t,\bx) \in \cT^N_{r_1^{(M)},R_M}} \Big( \cU^N(t,\bx)-\cV^N(t,\bx) \Big) \leq C/N,
    \end{align*}
   which together with Proposition \ref{prop.easyinequality} completes the proof.

\end{proof}

\section{The proof of Theorem \ref{thm.gradient}} \label{sec:gradientsconverge}

In this section we prove the convergence result for the gradients. Once again, we fix a $p > 2$ throughout the section. As in the proof of Theorem \ref{thm.main}, a useful first step is to observe that it suffices to prove the following proposition.

\begin{prop} \label{prop.tubegradient}
Let $(t_0,m_0) \in \cO$ and $R > 0$. Then there exists $r, C > 0$ such that, for each $N \in \N$, $i = 1,...,N$ and $(t,\bx) \in \cT^N_{r,R}(t_0,m_0)$, 
\begin{align*}
    |ND_{x^i} \cV^N(t,\bx) - D_{m} \cU(t,m_{\bx}^N)| \leq \frac{C}{N}.
\end{align*}
\end{prop}

It is easy to see that Proposition \ref{prop.tubegradient} implies Theorem \ref{thm.gradient}.

\begin{proof}[Proof of Theorem \ref{thm.gradient}] Theorem \ref{thm.gradient} follows from Proposition \ref{prop.tubegradient} exactly as Theorem \ref{thm.main} followed from Proposition \ref{prop.tube}.

\end{proof}

The rest of this section is devoted to proving Proposition \ref{prop.tubegradient}. Before we proceed, we fix some of the notations that will be used throughout the section.
\vs
First we define
\begin{align*}
    \cV^{N,i} = D_{x^i} \cV^N \ \ \text{and} \ \ \cU^{N,i} = D_{x^i} \cU^N \text{ for }i\in \{1,\cdots, N\}.
\end{align*}
Using the the PDEs for $\cV^N$ and $\cU^N$, it is easy to check that $\cV^{N,i}$ satisfies, in  $[0,T) \times (\R^d)^N$,
\begin{equation} \label{eq.vni}
\begin{split}
    &- \partial_t \cV^{N,i} - \sum_j \Delta_{x^j} \cV^{N,i} + \sum_j D_pH(x^j, N V^{N,j}) D_{x^j} \cV^{N,j} \\
    & \qquad \qquad + \frac{1}{N} D_{x^i} H(x^i, N \cV^{N,i}) = \sF^{N,i} \coloneqq \frac{1}{N} D_{m} \sF(t,m_{\bx}^N,x^i).
\end{split}
\end{equation}
Meanwhile, if $\cU$ were smooth enough on some tube $\cT_{r} = \cT_{r}(t_0,m_0)$, then we would have 
\begin{equation} \label{eq.uni}
\begin{split}
    &- \partial_t \cU^{N,i} - \sum_j \Delta_{x^j} \cU^{N,i} + \sum_j D_pH(x^j, N \cU^{N,j}) D_{x^j} \cU^{N,i} \\
    & \qquad  \qquad + \frac{1}{N} D_{x^i} H(x^i, N \cU^{N,i}) = \sF^{N,i} + E^{N,i} \ \ \text{in} \ \  \cT^N,
\end{split}
\end{equation} 
with 
\begin{align*}
    E^{N,i}(t,\bx) = -\frac{2}{N^2} D_x \tr \big[ D_{mm} \cU(t,m_{\bx}^N,x^i,x^i) \big] - \frac{1}{N^3} \sum_{j = 1}^N D_m \big[ \tr(D_{mm} \cU(t,m_{\bx}^N,x^j,x^j) \big](x^j).
\end{align*}
Of course, Theorem \ref{thm.c2} does not give enough regularity to immediately justify the computation above, since  it only shows that $D_{mm} \cU$ Lipschitz in an appropriate local sense. Nevertheless, when combining the regularity results Theorem \ref{thm.c2} and Proposition \ref{prop.timereg} with the reasoning in the proof of \cite[Proposition 5.1]{ddj2023}, one easily obtains the following analogue of Lemma \ref{lem:uneqn}.

\begin{lem} \label{lem:unieqn}
For any $(t_0,m_0) \in \cO$, there is $r > 0$ such that $\ov{\cT_{r}(t_0,m_0)} \subset \cO$, and, for each $N \in \N$, the projection $\cU^{N,i}$ defined by \eqref{undef} is Lipschitz in time and $\cC^1$ in space on the set $\cT^N_{r}(t_0,m_0)$, with spatial derivatives $D_{x^j} \cU^{N,i}$ being Lipschitz continuous in $\bx$ and such that the $\linf-$function 
\begin{align*}
E^{N,i} \coloneqq - \partial_t \cU^{N,i} - \sum_j \Delta_{x^j} \cU^{N,i} + \sum_j D_pH(x^j, N \cU^{N,j}) D_{x^j} \cU^{N,i} + \frac{1}{N} D_{x^i} H(x^i, N \cU^{N,i}) - \sF^{N,i} 
\end{align*}
satisfies, for each $N \in \N$, the bound
\begin{align*}
    \|E^{N,i}(t,\bx) \|_{\linf(\cT_{r}^N(t_0,m_0))} \leq C/N.
\end{align*}
\end{lem}
\vs
The next Lemma gives an analogue of Lemma \ref{lem: ImprovedRateTubes}, but at the level of the gradients of $\cU^N$ and $\cV^N$. 

\begin{lem}\label{lem: ImprovedRateTubesgradientl2}
    Let $(t_0,m_0) \in \cO$ and $r>0$ be small enough so that the conclusion of Lemma \ref{lem:unieqn} holds. Moreover, let
    \[\tau_{r}^{N,t,\bx}=\inf\{s>t: (s,\bX_x^{N,t,\bx})\in (\cT_{r}^N(t_0,m_0))^c\}\wedge T.\]
    Then, there exists an independent of $N$  constant $C>0$ such that, for all $(t,\bx)\in \cT_{r}^N$,
    \begin{align*}
     \sum_{i = 1}^N |\cU^{N,i}(t,\bx)-\cV^{N,i}(t,\bx)|^2 \leq \frac{C}{N^3} + \sup\limits_{(s,\by)\in \cT_{r}^N}\Big(\sum_{i = 1}^N |\cU^{N,i}(s,\by)-\cV^{N,i}(s,\by)|^2 \Big)\mathbb{P}[\tau_{r}^{N,t,\bx}<T].
    \end{align*}
\end{lem}

\begin{proof} 
For the sake of notational simplicity,  we give the proof for $d=1$. The general case follows the same steps. 
\vs
Fix $(t,\bx) \in \cT^{N}_{r}$ and consider  the optimal trajectory  $\bX = \bX^{(t,\bx)}$ for the $N$-particle control problem started from $(t,\bx)$.  Let $\tau = \tau_{r}^{N,t,\bx}$, drop the superscript $N$ and write  $E_s^{i}$ for $E_s^{N,i}$  and  set 
    \begin{align*}
        Y^i_s = \cV^{N,i}(s,\bX_s), \quad  Z^i_s = \sqrt{2} D \cV^{N,i}(s,\bX_s), \\
        \ov{Y}_s^i = \cU^{N,i}(s, \bX_s), \quad \ov{Z}_s^i = \sqrt{2} D \cU^{N,i}(s, \bX_s).
    \end{align*}
    On the stochastic interval $[t, \tau)$, we have the dynamics
    \begin{align*}
        dY_s^i =  \bigg(\frac{1}{N} D_{x^i} H(X_s^i, N Y_s^i) - \sF^{N,i}(s,\bX_s) \bigg) ds + Z_s^i d\bW_s, 
      \end{align*}      
and
      \begin{align*}  d\ov{Y}_s^i &= \bigg(\frac{1}{N} D_{x^i} H(X_s^i, N\ov{Y}_s^i) - \sF^{N,i}(s,\bX_s) - E^{N,i}(s,\bX_s)\\
        &\qquad  + \sum_j \Big(D_p H(X_s^j, N\ov{Y}^{j}) - D_p H(X_s^j, NY^{j}) \Big)D_{x^j} \cU^{N,i}(s,\bX_s) \bigg) ds + \ov{Z}_s^i d \bW_s, 
    \end{align*}
where we have written $\bW = (W^1,...,W^N)$ for simplicity.
\vs
We set 
\begin{align*}
    \Delta Y_s^i = Y_s^i - \ov{Y}_s^i \ \ \text{and} \ \ \Delta Z_s^i = Z_s^i - \ov{Z}_s^i,
\end{align*}
and, for each $i,j \in \{1,\cdots, N\}$, define the process $A^{i,j}$ by 
\begin{align*}
    A^{i,j}_s &= \frac{1}{|\Delta Y^{j}|^2} D_{x^j} \cU^{N,i}(s,\bX_s) \Big(D_p H(X_s^j, N Y^{j}) - D_p H(X_s^j, N\ov{Y}^{j}) \Big) \big(\Delta Y^{j}\big)^T 1_{\Delta Y^{j} \neq 0} \\
    &\qquad \qquad + \frac{1}{N|\Delta Y^{i}|^2} \Big(D_{x^i} H(X_s^i, NY_s^i) - D_{x^i}H(X_s^i, N \ov{Y}_s^i) \Big) \big(\Delta Y_s^i\big)^T 1_{\Delta Y^{i} \neq 0} 1_{i = j}.
\end{align*}
We may now write 
\begin{align*}
    d \Delta Y_s^i = \bigg(\sum_{j = 1}^N A_s^{i,j} \Delta Y_s^j + E^{i}_s \bigg)ds + \Delta Z_s^i d \bW_s.
\end{align*}
 The key point about the coefficients $A^{i,j}$ is that, for some $C$ independent of $N$,
\begin{align*}
    |A^{i,j}| \leq C/N + C 1_{i = j}.
\end{align*}
Next, we compute 
\begin{align*}
    d \big( \sum_{i=1}^N |\Delta Y_s^i|^2 \big) = 2 \big(\sum_{i,j = 1}^N A_s^{i,j}  \Delta Y_s^i  \Delta Y_s^j + \sum_{i = 1}^N \Delta Y_s^i E_s^i + \frac{1}{2} \sum_{i = 1}^N |\Delta Z_s^i|^2 \big) ds +2 \sum_{i=1}^N \Delta Y_s^i \Delta Z_s^i d \bW_s. 
\end{align*}
Integrating and taking expectations, we find, for some,  independent of $N$, $C$,
\begin{align*}
    \E\Big[ \sum_{i=1}^N |\Delta Y_{s \wedge \tau}^i|^2 \Big] &= \E\Big[\sum_{i=1}^N |\Delta Y_{\tau}^i|^2 \Big] - 2\bE\bigg[\int_{s \wedge \tau}^{\tau}\big(\sum_{i,j = 1}^N A_r^{i,j} \Delta  Y_r^i \Delta  Y_r^j + \sum_{i = 1}^N  \Delta Y_r^i E_r^i + \frac{1}{2} \sum_{i = 1}^N |\Delta Z_r^i|^2 \big) dr  \bigg] \\
    &\leq \E\Big[\sum_{i=1}^N |\Delta Y_{\tau}^i|^2 \Big] + C \E\bigg[\int_{s \wedge \tau}^{\tau} \Big( \sum_{i=1}^N |\Delta Y_{r}^i|^2 + \sum_{i=1}^N |E_{r}^i|^2 \Big) dr  \bigg] 
    \\
    &\leq \E\Big[\sum_{i=1}^N |\Delta Y_{\tau}^i|^2 \Big] + C \E\bigg[\int_{s}^{T} \Big( \sum_{i=1}^N |\Delta Y_{r \wedge \tau}^i|^2 + \sum_{i=1}^N |E_{r \wedge \tau}^i|^2 \Big) dr  \bigg]  \\
    &= \E\Big[\sum_{i=1}^N |\Delta Y_{\tau}^i|^2 \Big] + C \int_s^T \E\Big[ \sum_{i=1}^N |\Delta Y_{r \wedge \tau}^i|^2 + \sum_{i=1}^N |E_{r \wedge \tau}^i|^2 \Big] dr,
\end{align*}
with the first bound coming from Young's inequality and the bounds on $A^{i,j}$. 
\vs
Applying Gr\"onwall's inequality to the function $\psi(s) = \E\big[ \sum_i |\Delta Y_s^i|^2 \big]$ and using the bounds on $E^i$, we get 
\begin{align*}
   \sum_{i=1}^N |\cU^{N,i}(t,\bx) - \cV^{N,i}(t,\bx)|^2 &= \E[\sum_{i=1}^N |\Delta Y_t^i|^2 ] 
   \leq \frac{C}{N^3} + C \E\Big[\sum_{i=1}^N |\Delta Y_{\tau}^i|^2 \Big]
   \\
   & = \frac{C}{N^3} + C\E\Big[\sum_{i=1}^N |\cV^{N,i}(\tau, \bX_{\tau}) - \cU^{N,i}(\tau, \bX_{\tau})|^2 \Big] \\
   &\leq \frac{C}{N^3} + \sup\limits_{(s,\by)\in \cT_{r}^N}\Big(\sum_{i = 1}^N |\cU^{N,i}(s,\by)-\cV^{N,i}(s,\by)|^2 \Big)\mathbb{P}[\tau< T],
\end{align*}
which completes the proof.

\end{proof}
The following Lemma can be proved using Lemma \ref{lem: ImprovedRateTubesgradientl2}, exactly as Proposition \ref{prop.tube} is proved using Lemma \ref{lem: ImprovedRateTubes}. 

\begin{lem} \label{lem.tubegradientd2}
Let $(t_0,m_0) \in \cO$ and $R > 0$. Then there exist $r, C > 0$ such that,  for each $N \in \N$ and each $(t,\bx) \in \cT^N_{r,R}(t_0,m_0)$, 
\begin{align*}
    \sum_{i = 1}^N |ND_{x^i} V^N(t,\bx) - D_{m} U(t,m_{\bx}^N)|^2 \leq \frac{C}{N}.
\end{align*}
\end{lem}

We may now proceed with the proof of Proposition \ref{prop.tubegradient}.
\begin{proof}[Proof of Proposition~\ref{prop.tubegradient}]
Fix $(t_0,m_0) \in \cO$ and $R > 0$ and let $r>0$ be given by Lemma \ref{lem.tubegradientd2}. With  the notation of Lemma \ref{lem: ImprovedRateTubesgradientl2}, we now have, for each $i\in \{1,\cdots, N\}$, the dynamics
    \begin{equation}
        d\Delta Y_s^i = \Big(\sum\limits_{j}^NA_s^{i,j}\Delta Y_s^j +E_s^i \Big)ds +\Delta Z_s^i d\bW_s,
    \end{equation}
    Therefore,
\begin{align*}
    d \big(  |\Delta Y_s^i|^2 \big) = \big(\sum_{j = 1}^N A_s^{i,j} \Delta Y_s^i \Delta Y_s^j +  \Delta Y_s^i E_s^i + \frac{1}{2}  |\Delta Z_s^i|^2 \big) ds +  \Delta Y_s^i \Delta Z_s^i d \bW_s, 
\end{align*}
which yields
\begin{align*}
    \E\Big[  |\Delta Y_{s \wedge \tau}^i|^2 \Big] &= \E\Big[ |\Delta Y_{\tau}^i|^2 \Big] - \bE\bigg[\int_{s \wedge \tau}^{\tau}\big(\sum_{j = 1}^N A_r^{i,j}\Delta  Y_r^i\Delta Y_r^j + \Delta Y_r^i E_r^i + \frac{1}{2}  |\Delta Z_r^i|^2 \big) dr  \bigg] \\
    &\leq \E\Big[ |\Delta Y_{\tau}^i|^2 \Big] - \bE\bigg[\int_{s \wedge \tau}^{\tau}\big(\sum_{j \neq i}^N A_r^{i,j}\Delta  Y_r^i\Delta Y_r^j + A^{i,i}\Delta Y_r^{i}\Delta Y_r^{i}+ \Delta Y_r^i E_r^i \big) dr  \bigg]\\
    &\leq \E\Big[ |\Delta Y_{\tau}^i|^2 \Big] + C\E\bigg[\int_{s \wedge \tau}^{\tau} \Big( |\Delta Y_{r}^i|^2+N\sum_{j\neq i}|A^{i,j}|^2 |\Delta Y_r^j|^2 + |E_{r}^i|^2 \Big) dr  \bigg] 
    \\
    &\leq \E\Big[|\Delta Y_{\tau}^i|^2 \Big] + C \E\bigg[\int_{s}^{T} \Big( |\Delta Y_{r \wedge \tau}^i|^2 +\frac{1}{N}\sum\limits_{j\neq i}|\Delta Y_{r\wedge \tau}^j|^2 +  \frac{1}{N^4}\Big) dr  \bigg]  \\
    & \leq  \E\Big[ |\Delta Y_{\tau}^i|^2 \Big] + C \int_s^T \E\Big[ |\Delta Y_{r \wedge \tau}^i|^2 + \frac{1}{N^4} \Big] dr ,
\end{align*}
where in the second inequality we used Young's inequality in the form 
\[A_r^{i,j}\Delta  Y_r^i\Delta Y_r^j \leq \frac{1}{N}|\Delta Y_r^i|^2+CN|A_r^{i,j}|^2|\Delta Y_r^j|^2,\]
in the third inequality we used the bounds on $|E_r^i|$ as well as the bounds on $A^{i,j}$ for $i\neq j$, and,  finally, in the last inequality we used Lemma \ref{lem.tubegradientd2}. 

\vs
Thus by Gronwall and the fact that 
\begin{align*}
    |\Delta Y_{\tau}^i|^2 \leq \begin{cases}
    0 &  \text{if} \ \  \tau = T, \\
    \sup_{(s,\by) \in \cT_{r}^N} \Big(|\cU^{N,i}(s,\by)-\cV^{N,i}(s,\by)|^2 \Big) & \text{if} \ \  \tau < T,
    \end{cases}
\end{align*}
we have that 
\[|\cU^{N,i}(t,\bx)-\cV^{N,i}(t,\bx)|^2 \leq \frac{C}{N^{4}}+\sup\limits_{(s,\by)\in \cT_{r}^N}\Big(|\cU^{N,i}(s,\by)-\cV^{N,i}(s,\by)|^2 \Big)\mathbb{P}[\tau_{r}^{N,t,\bx}<T]. \]
To conclude,  we follow the same procedure as in the proof of Proposition \ref{prop.tube} to show that there exists a $r_0<r$ such that 
\[\sup\limits_{(s,\by)\in \cT_{r_0}^N}\Big(|\cU^{N,i}(s,\by)-\cV^{N,i}(s,\by)|^2 \Big) \leq \frac{C}{N^4},\]
which proves the result.

\end{proof}

\section{The proof of Proposition \ref{prop.concentration}} \label{sec:concentration}

The proof of Proposition \ref{prop.concentration} follows from the following two Lemmas.

\begin{lem} \label{lem.concentration}
 Let Assumption \ref{assump:gradients} hold. Fix $(t_0,m_0) \in \cO$,  assume that $m_0$ satisfies the condition \eqref{t2} from the statement of Proposition \ref{prop.concentration}, and,
 for each $N \in \N$, denote by $\mathbf{\tilde{X}}^N = (\tilde{X}^{N,1},...,\tilde{X}^{N,N})$ the solution to 
 \begin{equation}\label{ynsymmetric}
    \begin{split}
         &d\tilde{X}_t^{N,i}  = - D_p H(\tilde{X}^{N,i}, D_m U(t, m_{\mathbf{\tilde{X}}_t^N}^N, \tilde{X}_t^{N,i})) dt \\
         & \qquad + \sqrt{2} dW_t^i \ \ \text{ for} \ \  t_0 \leq t \leq \tau =\inf \Big\{ t > t_0 : (t,m_{\tilde{\bX}_t}^N) \in \cO^c \Big\} \wedge T, \quad \\
        & \tilde{X}_{t_0}^{N,i} = \xi^i,
     \end{split}
 \end{equation}
    where $(\xi^i)_{i \in \N}$ are i.i.d. with common law $m_0$. Then, for each $r > 0$, there exist  constants $c,C > 0$ such that,   for each $N \geq 1$, 
    \begin{align*}
        \P\Big[ \sup_{t_0 \leq t \leq \tau} \bd_2\big(m_{\tilde{\bX}_{t}^N}^N, m_t^{(t_0,m_0)} \big)  > r \Big] \leq Ce^{-cN}.
    \end{align*}
\end{lem}

\begin{proof}
    By Proposition \ref{prop.dmlip} below, we can choose $r_0$ small enough so  that $\ov{\cT_{r_0}(t_0,m_0)} \subset \cO$ and $(x,m) \mapsto - D_p H\big(x,D_m\cU(t,m,x) \big)$ is uniformly Lipschitz on $\cT_r(t_0,m_0)$ with $m$ endowed with the $\bd_2-$metric. We then extend, as in the proof of Lemma \ref{lem.key}, to find a measurable map 
    \begin{align*}
        b(t,x,m) : [0,T] \times \R^d \times \cP_2(\R^d) \to \R^d
    \end{align*}
    which is globally Lipschitz in $(x,m)$ and such that 
    \begin{align*}
        b(t,x,m) = - D_p H(x, D_mU(t,m,x)) \ \ \text{for} \ \  (t,m) \in \cT_{r_0}. 
    \end{align*}
    Let $\bY^N = (\bY^{N,1},...,\bY^{N,N})$ be the unique solution on $[t_0,T]$ of the SDE
    \begin{equation}
         dY^{N,i}_t = b(t,Y^{N,i}_t,m_{\bY_t^N}^N)  dt + \sqrt{2} dW_t^i \ \ \text{in} \ \  (t_0 ,T] \ \ \text{and} \ \  Y_{t_0} = \xi^i,
 \end{equation}
and   notice that $\bY = \tilde{\bX}$ on $[0,\sigma)$, where $\sigma = \inf \{t \geq t_0 : (t,m_{\bY_t}^N) \notin \cT_{r}\}$. 
\vs
Assuming without loss of generality that $r < r_0$, we can thus use \cite[Theorem 5.4]{Delarue2018FromTM} to conclude
    \begin{align*}
       \P\Big[ \sup_{t_0 \leq t \leq \tau} \bd_2\big(m_{\tilde{\bX}_{t}^N}^N, m_t^{(t_0,m_0)} \big)  > r \Big] = \P\Big[ \sup_{t_0 \leq t \leq \tau} \bd_2\big(m_{\bY_{t}^N}^N, m_t^{(t_0,m_0)} \big)  > r \Big] \leq Ce^{-cN}.
    \end{align*}
\end{proof}

\begin{lem} \label{lem.xycomp}
    Let $\tilde{X}^{N,i}$ and $X^{N,i}$ be defined by  \eqref{def.xnsymmetric} and \eqref{ynsymmetric} respectively, fix $R > 0$,   choose $r$  small enough so that the conclusion of Proposition \ref{prop.tubegradient} holds on $\cT_{r}(t_0,m_0)$, and let
    \begin{align} \label{sigmadef}
        \sigma = \inf \big\{t > t_0 : (t,m_{\bY_t^N}^N) \notin \cT^{N}_{r, R}(t_0,m_0)) \text{ or } (t,m_{\bX_t^N}^N) \notin \cT^{N}_{r, R}(t_0,m_0))  \big\} \wedge T.
    \end{align}
    Then there exists a constant $C > 0$ such that, a.s., 
    \begin{align*}
        \sup_{t_0 \leq t \leq \sigma} \Big\{\frac{1}{N} \sum_{i = 1}^N |X_t^{N,i} - Y_t^{N,i}| \Big\} \leq C/N. 
    \end{align*}
\end{lem}

\begin{proof}
For simplicity of notation, we write 
\begin{align*}
    b^{N,i}(t,\bx) = - D_p H(x^i, ND_{x^i} \cV^N(t,\bx)) \ \ \text{and} \ \ \tilde{b}^{N,i}(t,\bx) = - D_p H(x^i, D_m \cU(t,m_{\bx}^N,x^i)).
\end{align*}
On $[t_0,\sigma)$ we can rewrite the dynamics of $\bX^N$ as 
\begin{align*}
dX_t^{N,i} = b^{N,i}(t,\bX_t^{N,i}) dt + \sqrt{2} dW_t^i = \Big(\tilde{b}^{N,i}(t,\bX_t^{N,i}) + E_t^i\Big) dt + \sqrt{2} dW_t^i, 
\end{align*}
where,  in view or Theorem~\ref{thm.gradient}, 
\begin{align*}
    |E_t^i| = |b^{N,i}(t,\bX_t^{N,i}) - \tilde{b}^{N,i}(t,\bX_t^{N,i})| \leq C/N. 
\end{align*}
Thus setting $\Delta X_t^{N,i} = X_t^{N,i} - \tilde{X}^{N,i}_t$, we have, for $t_0 \leq t \leq \sigma$, 
\begin{align*}
    \Delta{X}_t^{N,i} = \int_{t_0}^t \Big(\tilde{b}^{N,i}(t,X_t) - \tilde{b}^{N,i}(t,\tilde{X}_t) + E_t^{N,i} \Big). 
\end{align*}
Using the bounds on $E^{N,i}$ and the fact that the regularity of $U$ implies $D_{x^j} \tilde{b}^{N,i} \leq C/N + C 1_{i = j}$, we easily get
\begin{align*}
    \frac{1}{N} \sum_{i = 1}^N |\Delta X_t^{N,i}|^2 \leq C/N^2 + C\int_{t_0}^t \frac{1}{N} \sum_{i = 1}^N |\Delta X_s^{N,i}|^2 ds, 
\end{align*}
and we conclude with Gronwall's inequality. 

\end{proof}

\begin{proof}[The proof of Proposition \ref{prop.concentration}]
    We fix $R > 0$ to be chosen later, without loss of generality assume that $r$ is small enough so that the conclusion of Theorem~\ref{thm.gradient} holds on $\cT_{r, R}(t_0,m_0)$, and 
let $\tilde{\bX}^N$ and $\sigma$ be defined as \eqref{ynsymmetric} and  \eqref{sigmadef} respectively. 
\vs
We note that  to prove Proposition \ref{prop.concentration}, it suffices to show an estimate of the form $$\bP[\sigma < T] \leq C \exp(-cN^{1- \eta}).$$
 To this end, we remark  that
    \begin{align*}
        \sigma \geq \sigma_R \wedge \sigma_T \wedge \tilde{\sigma}_R \wedge \tilde{\sigma}_T \wedge T,
    \end{align*}
    where 
    \begin{align*}
        \sigma_R = \inf\{t \geq t_0 : (t_0,m_{\bX^N_t}^N) \notin B^p_R \} \ \ \text{and} \ \  \tilde{\sigma}_R = \inf\{t \geq t_0 : (t_0,m_{\tilde{\bX^N}_t}^N) \notin B^p_R \}, 
       \end{align*}
      and    
      \begin{align*}
        \sigma_{\cT} = \inf\{t \geq t_0 : (t_0,m_{\bX^N_t}^N) \notin \cT_{r} \}  \ \ \text{and} \ \  \tilde{\sigma}_{\cT} = \inf\{t \geq t_0 : (t_0,m_{\tilde{\bX}^N_t}^N) \notin \cT_{r/2} \}.
    \end{align*}
    Then we have 
    \begin{align*}
       \P&\Big[ \sup_{t_0 \leq t \leq T} \bd_2\big(m_{\bX_t^N}^N, m_t^{(t_0,m_0)} \big)  > r \Big] 
    \leq  \P[ \sigma_R < T] + \P[\tilde{\sigma}_R < T] + \P[\tilde{\sigma}_{\cT} < T] \\
    &\qquad \qquad +  \P \Big[ \sup_{t_0 \leq t \leq T} \bd_2\big(m_{\bX_t^N}^N, m_t^{(t_0,m_0)} \big)  > r  \text{ and } \sigma_R = \tilde{\sigma}_R = \tilde{\sigma}_{\cT} = T\Big] \\
    & \qquad \qquad \leq \P[ \sigma_R < T] + \P[\tilde{\sigma}_R < T] + \P[\tilde{\sigma}_{\cT} < T] + \bP\Big[\sup_{t_0 \leq t \leq \sigma} \bd_2\big(m_{\bX_t^N}^N, m_{\tilde{\bX}_t^N}^N \big) > \frac{r}{2} \Big].
   \end{align*}
  By Lemma \ref{lem.xycomp}, the last term in the final line above vanishes when $N$ is large enough. 
 \vs
 
  To bound $\bP[\sigma_R < T]$, we argue as in the proof of Lemma \ref{lem.concentration} to conclude that 
  \begin{align*}
      \sup_{t_0 \leq t \leq T} \bd_p(m_{\bX_t}^N, m_{\bx_0}^N) \leq C \Big(1 + \frac{1}{N} \sum_{i = 1}^N \sup_{t_0 \leq t \leq T} |W_t^i- W_{t_0}^i|^p \Big). 
  \end{align*}
  We conclude that
  \begin{align*}
      \P[\sigma_R < T] \leq \P\Big[C \Big(1 + \frac{1}{N} \sum_{i = 1}^N \sup_{t_0 \leq t \leq T} |W_t^i- W_{t_0}^i|^p \Big) > R^p - C \Big], 
  \end{align*}
  and so, using a classical concentration inequality for ``sub-Weybull" random variables which can be traced back to \cite{Nagaev}, we have, for $R$ large enough, $$\P[\sigma_R < T] \leq C \exp(- c N^{2/p}).$$
   An identical argument shows the same estimate for $\bP[\tilde{\sigma}_R < T]$.
\vs
    Finally, by Lemma \ref{lem.concentration}, we have $$\bP[\tilde{\sigma}_{\cT} < T] \leq C\exp(-cN).$$ Since $p > 2$ is arbitrary, this completes the proof.

\end{proof}

\section{Regularity } \label{sec:regularity}
In this section we show the necessary regularity results for the value function $U$. 
\subsection{Terminology and notation}  Throughout this part, we set 
\[
F(x,m)= \frac{\delta \mathcal F}{\delta m}(m,x) \ \ \text{and} \ \  G(x,m)= \frac{\delta \mathcal G}{\delta m}(m,x).
\]
We also make use of the notation

\begin{align*}
    \frac{\delta F}{\delta m}(x,m)(\rho) =  \langle \frac{\delta F}{\delta m}(x,m,y) ,\rho(dy)\rangle, \quad \frac{\delta^2 F}{\delta m^2}(x,m)(\rho_1)(\rho_2) = \langle \langle \frac{\delta^2 F}{\delta m^2}(x,m,y,z) ,\rho_1(dy)\rangle,\rho_2(dz)\rangle, 
\end{align*}
whenever $\rho$, $\rho_1$, $\rho_2$ are distributions such that the above pairings make sense, and similar notations with $G$ replacing $F$.
\vs
As in \cite{CardSoug2022}, we need to study a number of linear equations, which are obtained by linearizing the MFG system  describing the optimal trajectories. For the reader's convenience, we list here all of the relevant  equations.
\newline \newline 
\noindent We say that $(u,m)$ is a solution to \eqref{system: MFGSystem} with initial condition $m(t_0)=m_0$, if
\begin{equation}\label{system: MFGSystem}\tag{MFG}
\begin{cases}
 -\pt u -\Delta u+H(x,Du)=F(x,m(t)) \ \text{ in } \ (t_0,T)\times \R^d,\\[1mm]
    \pt m-\Delta m-\dive(mD_pH(x,Du))=0 \ \text{ in } \ \text{ in }(t_0,T)\times \R^d,\\[1mm]
    m(t_0,\cdot)=m_0 \ \text{and} \ u(T,\cdot)=G(\cdot,m(T)) \ \text{ in } \ \R^d.
\end{cases}
\end{equation}
We say that $(z,\rho)$ is a solution to \eqref{system: LinearizedSystemNoError}  driven by $(u,m)$ and with initial condition $\rho(t_0)=\xi $,  if 
\begin{equation}\label{system: LinearizedSystemNoError}\tag{MFGL}
    \begin{cases}
    -\pt z-\Delta z+D_pH(x,Du)\cdot Dz=\dfrac{\delta F}{\delta m}(x,m(t))(\rho(t)) \ \text{ in } \ (t_0,T)\times \R^d,\\[1.2mm]
    \pt \rho -\Delta\rho -\dive(\rho D_pH(x,Du))=\dive(m  D_{pp}H(x,Du)Dz) \ \text{ in } \ (t_0,T)\times \R^d,\\[1mm]
    \rho(t_0,\cdot)=\xi  \ \text{and} \  z(T,\cdot)=\dfrac{\delta G}{\delta m}(\cdot,m(T))(\rho(T)) \ \text{ in } \ \R^d.
    \end{cases}
\end{equation}
We say that $(z,\rho)$ is a solution to \eqref{system: LinearizedSystemWithError} driven by $(u,m)$  with initial condition $\rho(t_0)=\xi $ and forcing terms $R^1,R^2,R^3$ if 
\begin{equation}\label{system: LinearizedSystemWithError}\tag{MFGLE}
    \begin{cases}
    -\pt z-\Delta z+D_pH(x,Du)\cdot Dz=\dfrac{\delta F}{\delta m}(x,m(t))(\rho )+R^1 \ \text{ in } \ (t_0,T)\times \R^d,\\[1.2mm]
    \pt \rho -\Delta \rho -\dive(\rho D_pH(x,Du))=\dive(mD_{pp}H(x,Du)Dz)\\[1.2mm]
    \hskip2.5in  +\dive(R^2) \ \text{ in } \ (t_0,T)\times \R^d,\\[1mm]
    \rho(t_0,\cdot)=\xi \ \text{and} \  z(T,\cdot)=\dfrac{\delta G}{\delta m}(\cdot,m(T))(\rho (T))+R^3 \ \text{ in } \ \R^d.

    \end{cases}
\end{equation}
Given a sufficiently smooth vector field $V:[0,T]\times \R^d\to \R^d$ and a bounded map $\Gamma:[0,T]\times \R^d\to \R^{d\times d}$, we say that that $(z,\rho)$ is a solution to \eqref{system: LinearizedSystemGeneralForm} driven by $(u,m)$ with initial condition $\rho(t_0)=\xi$ and forcing terms $R^1,R^2,R^3$ if 
\begin{equation}\label{system: LinearizedSystemGeneralForm}\tag{MFGLG}
    \begin{cases}
    -\pt z-\Delta z+V(t,x)\cdot Dz=\dfrac{\delta F}{\delta m}(x,m(t))(\rho )+R^1 \ \text{ in } \ (t_0,T)\times \R^d,\\[1.2mm]
    \pt \rho -\Delta \rho -\dive(\rho V)=\sigma\dive(m\Gamma Dz)+\dive(R^2) \ \text{ in } \ (t_0,T)\times \R^d,\\[1mm]
    \rho(t_0,\cdot)=\xi \ \text{and} \  z(T,\cdot)=\dfrac{\delta G}{\delta m}(\cdot,m(T))(\rho (T))+R^3 \ \text{ in } \ \R^d,
    \end{cases}
\end{equation}
where $m$ solves 
\begin{equation}\label{aux: GeneralLinearizedSystem}
\begin{cases}
    \pt m -\Delta m-\dive(m V)=0 \ \text{ in } \ (t_0,T)\times \R^d,\\[1mm]
    m(t_0,\cdot) =m_0 \  \text{ in } \ \R^d.
\end{cases}
\end{equation}
\vs
We provided separate definitions for the above systems due to their frequent use. We note, however, that $\ref{system: LinearizedSystemNoError}$ is a special case of $\ref{system: LinearizedSystemWithError}$, which in turn is a special case of \ref{system: LinearizedSystemGeneralForm}. 
\vs
Finally, we recall the notion of strong stability used in \cite{CardSoug2022} for the system
\begin{equation}\label{system: StrongStabil}
    \begin{cases}
        -\pt z-\Delta z+V(t,x)\cdot Dz=\dfrac{\delta F}{\delta m}(x,m(t))(\rho ) \  \text{ in } \ (t_0,T)\times \R^d,\\[1.2mm]
    \pt \rho -\Delta \rho -\dive(\rho V)=\sigma\dive(m\Gamma Dz) \ \text{ in } \ (t_0,T)\times \R^d,\\[1mm]
    \rho(t_0,\cdot)=\xi  \ \text{ and } \ z(T,\cdot)=\dfrac{\delta G}{\delta m}(\cdot,m(T))(\rho (T)) \  \text{ in } \ \R^d.
    \end{cases}
\end{equation}
We say that 
\begin{equation}\label{defn: stability}
\begin{split}
&\text{ the system \eqref{system: StrongStabil} is strongly stable if, for any }\sigma\in [0,1],\\
&\text{ its unique solution is }(z,\rho) = (0,0).
\end{split}
\end{equation}

\subsection{Refinement of the results in \cite{CardSoug2022}} The purpose of this subsection is to present sharpened versions of the results in \cite{CardSoug2022} under the increased regularity of the data. The majority of them require only small adjustments. The only critical extensions are Lemma \ref{lem: MainEstimatesLinearizedSystem} for estimates on the linearized system \ref{system: LinearizedSystemWithError} and Lemma \ref{lem: MainEstimatesFOrDifferenceOfSOlutions} for the stability of controls. For this reason we include detailed proofs of these two results.
\vs
The following is a generalization of \cite[Lemma 2.1]{CardSoug2022}. The only improvement is in the norm of dependence on $\xi$.
\begin{lem}\label{lem: MainEstimatesLinearizedSystem}
Assume \eqref{assump:gradients} and \eqref{defn: stability}. There exists a neighborhood $\mathcal{V}$ of $(V,\Gamma)$ in the topology of locally uniform convergence, and $\eta,C>0$ such that, for any $(V',t_0',\Gamma',R^{1,'},R^{2,'},R^{3,'},\xi',\sigma ')$ with 
\begin{equation}
\begin{cases}
    (V',\Gamma')\in \mathcal{V},\,\, |t_0'-t_0|+d_2(m_0',m_0)\leq \eta,\,\, \|V'\|_{C^{1,3}}+\|\Gamma '\|_{\infty}\leq 2C_0, \,\, \sigma ' \in [0,1], \vspace{.2cm} \\
    R^{1,'}\in C^{\delta/2,\delta} ,\,\, R^{2,'}\in L^{\infty}([t_0,T],(W^{1,\infty})'(\R^d,\R^d)), R^{3,'}\in C^{2+\delta}, \,\, \xi\in (C^{1+\delta})',
    \end{cases}
\end{equation}
any solution $(z',\rho')$ to \eqref{system: LinearizedSystemGeneralForm} associated with these data on $[t_0',T]$ and $m'$ the solution to \eqref{aux: GeneralLinearizedSystem} with drift $V'$ and initial condition $m_0'$ at time $t_0'$ satisfies
\begin{equation}
    \|z'\|_{C^{(2+\delta)/2,2+\delta}}+\sup\limits_{t\in [t_0',T]}\|\rho'(t,\cdot)\|_{(C^{2+\delta})'}+ \sup\limits_{t'\neq t}\frac{\|\rho'(t',\cdot)-\rho'(t,\cdot)\|_{(C^{2+\delta})'}}{|t'-t|^{\delta/2}}\leq CM',
\end{equation}
where 
\begin{equation}
    M'=\|\xi'\|_{(C^{1+\delta})'}+\|R^{1,'}\|_{C^{\delta/2,\delta}}+\sup\limits_{t\in [t_0',T]}\|R^{2,'}(t)\|_{(W^{1,\infty})'}+\|R^{3,'}\|_{C^{2+\delta}}.
\end{equation}

\end{lem}
The proof is identical to the one in \cite{CardSoug2022}, where a careful inspection of the proofs of \cite[Lemma 2.3]{CardSoug2022} and  \cite[Lemma 2.1]{CardSoug2022} shows that we may in fact use $\|\xi\|_{(C^{1+\delta})'}$ instead of $\|\xi\|_{(W^{1,\infty})'}$. 
\vs
The following is a restatement of \cite[Lemma 1.3]{CardSoug2022} for measures in $\mathcal{P}_1(\R^d)$.
\begin{lem}\label{lem: EstimatesOnMFG}
    Assume \eqref{assump:gradients} and let $(u,m)$ be a solution of \ref{system: MFGSystem}. Then there exists $C>0$, which is independent of $(t_0,m_0)$, such that 
    \begin{equation}
        \|u\|_{C^{(3+\delta)/2, 3+\delta}}+\sup\limits_{t\neq t'}\frac{d_1(m(t),m(t'))}{|t'-t|^{\frac{1}{2}}}\leq C
        \end{equation}
         and 
          \begin{equation}
        \sup\limits_{t\in [t_0,T]}\intd |x|m(t,dx)\leq C\intd |x|m_0(dx).
    \end{equation}
\end{lem}
We now present the critical improvement of \cite[Lemma 2.9]{CardSoug2022}. 
\begin{lem}\label{lem: MainEstimatesFOrDifferenceOfSOlutions}
    Assume \eqref{assump:gradients} and fix $(t_0,m_0)\in \mathcal{O}$. There exist $\theta ,C>0$ such that, for any $t_0',m_0^1,m_0^2$ satisfying $ |t_0'-t_0|<\theta$ and $\dtwo(m_0,m_0^i)<\theta$, if $(m^i,\alpha^i)$ is the unique minimizer starting from $(t_0',m_0^i)$ with associated multiplier $u^i$ for $i=1$ and $i=2$, then 
    \begin{align*}
        &\|u^2-u^1\|_{C^{(2+\delta)/2,2+\delta}}+\sup\limits_{t\in [t_0',T]}\bd_1(m^2(t),m^1(t))\\
        &+\sup\limits_{t'\neq t}\frac{\|(m^2-m^1)(t')-(m^2-m^1)(t)\|_{(C^{2+\delta})'}}{|t'-t|^{\delta/2}}\leq C\bd_1(m_0^2,m_0^1).
\end{align*}
\end{lem}
\begin{rmk}
    We note that  in the statement of Lemma \ref{lem: MainEstimatesFOrDifferenceOfSOlutions} $m_0,m_0^i$ are close with respect to $\bd_2$  while the result uses $\bd_1(m_0^2,m_0^1)$. This is done to be consistent with the results already proven in \cite{CardSoug2022}, where the set $\mathcal{O}$ was shown to be open in the $\bd_2$-topology. However, as we show below the function $\cU$ is in fact regular in $\bd_1$ inside the set $\mathcal{O}$. The  use of the metrics $\bd_1$ and $\bd_2$ will appear throughout this section.
\end{rmk}
\begin{proof}
 Let $(m,\alpha)$ be the unique stable minimizer starting from $(t_0,m_0)$ with multiplier $u$. Then \cite[Lemma 2.6]{CardSoug2022}, yields that system \ref{system: LinearizedSystemNoError} is strongly stable. 
 \vs
 Let $V=-D_pH(x,Du), \Gamma = -D_{pp}H(x,Du)$ and $\mathcal{V}$ be the corresponding neighborhood as described in Lemma \ref{lem: MainEstimatesLinearizedSystem}. With the same argument as in \cite[Lemma 2.9]{CardSoug2022}, by choosing $\theta >0$ small enough, we have that
\[(V^i,\Gamma^i)\in \mathcal{V}\]
where $V^i=-D_pH(x,Du^i), \Gamma^i = -D_{pp}H(x,Du^i)$. 
\vs
Given $\eta>0$, by choosing $\theta>0$ even smaller if necessary,  we have that under our assumptions 
\begin{equation}\label{eq: ProximityWIthEta}
    \|u^2-u^1\|_{C^{(2+\delta)/2,2+\delta}}+\sup\limits_{t\in [t_0',T]}\dtwo(m^2(t),m^1(t))<\eta .
\end{equation}
Moreover, it is easy to check that, for a constant $C=C(T,H,\|D^2u^1\|_{\infty},\|D^2u^2\|_{\infty})>0$  which is bounded by Lemma \ref{lem: EstimatesOnMFG}, we have 
\begin{equation}\label{eq: D1Sup}
    \sup\limits_{t_0'\leq t\leq T}\bd_1(m^2(t),m^1(t))\leq C(\bd_1(m_0^2,m_0^1) +\|Du^2- Du^1)\|_{\infty}).
\end{equation}
Consider the pair 
\[(v,\rho )=(u^2-u^1,m^2-m^1),\]
which solves \eqref{system: LinearizedSystemWithError} with
\begin{align*}
    R^1(t,x) &= H(x,Du^2)-H(x,Du^1)-D_pH(x,Du^1)\cdot (Du^2-Du^1) \\
        &\qquad  +F(x,m^2)-F(x,m^1)-\frac{\delta F}{\delta m}(x,m^1(t))(m^2(t)-m^1(t)), \\[1mm]
    R^2(t,x) &=D_pH(x,Du^2)m^2-D_pH(x,Du^1)m^1-D_pH(x,Du^1)(m^2-m^1)\\
    &\qquad -D_{pp}H(x,Du^1)(Du^2-Du^1)m^1 \\
        &=\Big(D_pH(x,Du^2)-D_pH(x,Du^1)\Big)(m^2-m^1) 
        \\
        &\qquad  +\Big(D_pH(x,Du^2)-D_pH(x,Du^1)-D_{pp}H(x,Du^1)\cdot (Du^2-Du^1) \Big)m^1 
        \\
        &=\Big(D_pH(x,Du^2)-D_pH(x,Du^1)\Big)(m^2-m^1) \\ &\qquad +m^1\int_0^1 D(u^2-u^1)\cdot\Big( D_{pp}H(\lambda Du^2+(1-\lambda)Du^1)-D_{pp}H(x,Du^2)\Big)d\lambda, \\
    R^3(x,T) &=G(x,m^2(T))-G(x,m^1(T))-\frac{\delta G}{\delta m}(x,m^1(T))(m^2(T)-m^1(T)), \\
    \xi &=m_0^2-m_0^1.
\end{align*}
Note that we have the estimates
\begin{align*}
    \|\xi\|_{(C^{1+\delta})'} &\leq \sup\limits_{\|Df\|_{\infty}\leq 1}\int_{\R^d}f(x)(m_0^2-m_0^1) = \bd_1(m_0^2,m_0^1), \\
    \sup\limits_{t\in [t_0,T]}\|R^{2}(t)\|_{(W^{1,\infty})'} &\leq C\Big(\|Du^2-Du^1\|_{C^{0,2}}\sup\limits_{t\in [t_0,T]}\bd_1(m^2(t),m^1(t))+\|Du^2-Du^1\|_{C^{0,2}}^2\Big) \\
    &\leq C\Big(\|u^2-u^1\|_{C^{(2+\delta)/2,2+\delta}}^2+\sup\limits_{t\in [t_0,T]}\bd_1^2(m^2(t),m^1(t))\Big). \\
    \|R^3\|_{C^{2+\delta}} &\leq C\sup\limits_{t\in [t_0,T]}\bd_1^2(m^2(t),m^1(t)).
\end{align*}
It remains to estimate the quantity $\|R^1\|_{C^{\delta/2,+\delta}}$. For this, we  rewrite $R^1$ as  
\begin{align*}
    R^1(t,x) &= A(t,x) + B(t,x),
    \end{align*}
    where 
    \begin{align*}
    A(t,x) = -\int_0^1 \Big(D_pH(x,\lambda Du^2+(1-\lambda )Du^1)-D_pH(Du^1)\Big)\cdot (Du^2-Du^1)d\lambda, 
    \end{align*}
    and 
    \begin{align*}
    B(t,x)= \int_0^1 \Big(\frac{\delta F}{\delta m}(x,\lambda m^2(t)+(1-\lambda )m^1(t))-\frac{\delta F}{\delta m}(x,m^1(t))\Big)(m^2(t)-m^1(t)) d\lambda  \Big).
\end{align*}
\vs
Bounding $A$ is relatively straightforward, since
\[\|A\|_{C^{\delta/2,\delta}}\leq C\|D(u^2-u^1)\|_{C^{\delta/2,\delta}}^2\leq C\|u^2-u^1\|_{C^{\delta/2, 2+\delta}}^2.\]
\vs
The  bound $B$ is a bit more involved. Given $\lambda,\theta \in [0,1]$, let
\begin{align*}
    [m]_{\lambda}(t)=\lambda m^2(t)+(1-\lambda) m^1(t)), \quad [m]_{\lambda ,\theta}(t)=\theta [m]_{\lambda }(t)+(1-\theta)m^2(t), \quad  \rho(t)=m^2(t)-m^1(t).
\end{align*}
\vs
We rewrite $B$ as
\[B(t,x)= \int_0^1\int_0^1 \lambda \frac{\delta^2 F}{\delta m^2}\Big(x,[m]_{\lambda,\theta}(t) \Big)(\rho(t))(\rho(t))d\lambda d \theta,\]
and look at the functions 
\[C_{\lambda,\theta}(t,x) =\frac{\delta^2 F}{\delta m^2}\Big(x,[m]_{\lambda,\theta}(t) \Big)(\rho(t))(\rho(t)).\]
For every $\lambda,\theta \in [0,1]$, we have the estimates
\begin{align*}
    \|C_{\lambda,\theta}(\cdot,\cdot)\|_{\delta/2,\delta} &\leq C\bigg(\sup\limits_{t\neq s}\frac{\|[m]_{\lambda ,\theta}(t)-[m]_{\lambda ,\theta}(s)\|_{(C^{1})'}}{|t-s|^{\delta/2}}\sup\limits_{t\in [t_0,T]}\|\rho(t)\|_{(C^{2+\delta})'}^2\\
    & \quad +\sup\limits_{t\neq s}\frac{\|\rho(t)-\rho(s)\|_{(C^{2+\delta})'}}{|t-s|^{\delta/2}}\sup\limits_{t\in [t_0,T]}\|\rho(t)\|_{(C^1)'}\bigg)\\
    &\leq C\bigg( \Big(\sup\limits_{t\neq s}\frac{\bd_1(m^2(t),m^2(s))}{|t-s|^{\delta/2}}+\sup\limits_{t\neq s}\frac{\bd_1(m^1(t),m^1(s))}{|t-s|^{\delta/2}} \Big)\sup\limits_{t\in [t_0,T]}\bd_1^2(m^2(t),m^1(t))\\
    & \quad +\sup\limits_{t\in [t_0,T]}\bd_1(m^2(t),m^1(t))\frac{\|(m^2(t)-m^1(t))-(m^2(s)-m^1(s))\|_{(C^{2+\delta})'}}{|t-s|^{\delta/2}} \bigg)
\end{align*}
and thus 
\begin{align*}
    \|B(\cdot,\cdot)\|_{\delta/2,\delta}\leq& C\Big(\sup\limits_{t\in [t_0,T]}\bd_1(m^2(t),m^1(t))\frac{\|(m^2(t)-m^1(t))-(m^2(s)-m^1(s))\|_{(C^{2+\delta})'}}{|t-s|^{\delta/2}}\\
    & \qquad \qquad +\sup\limits_{t\in [t_0,T]}\bd_1^2(m^2(t),m^1(t))\Big).
\end{align*}
Combining the upper bounds  on $A$ and $B$, we finally conclude that
\[\|R^1\|_{C^{\delta/2,\delta}}\leq C\Big(\sup\limits_{t\in [t_0,T]}\bd_1(m^2(t),m^1(t))\frac{\|(m^2(t)-m^1(t))-(m^2(s)-m^1(s))\|_{(C^{2+\delta})'}}{|t-s|^{\delta/2}}\]
\[+\sup\limits_{t\in [t_0,T]}\bd_1^2(m^2(t),m^1(t))+\|u^2-u^1\|_{C^{\delta/2,2+\delta}}^2\Big),\]
and, hence, using  Lemma \ref{lem: MainEstimatesLinearizedSystem} we get 
\begin{align*}
    \|u^2-u^1&\|_{C^{(2+\delta)/2,2+\delta}}+\sup\limits_{t\in [t_0,T]}\|m^2(t)-m^1(t)\|_{(C^{2+\delta})'}+\sup\limits_{t\neq s}\frac{\|(m^2(t)-m^1(t))-(m^2(s)-m^1(s)\|_{(C^{2+\delta})'}}{|t-s|^{\delta/2}} \\
    &\leq C\bigg(\|u^2-u^1\|_{C^{\delta/2,2+\delta}}^2+\sup\limits_{t\in [t_0,T]}\bd_1^2(m^2(t),m^1(t)) \\
    &\qquad +\sup\limits_{t\in [t_0,T]}\bd_1(m^2(t),m^1(t))\frac{\|(m^2(t)-m^1(t))-(m^2(s)-m^1(s))\|_{(C^{2+\delta})'}}{|t-s|^{\delta/2}}\bigg).
\end{align*}
Thus choosing $\eta>0$ small enough in \eqref{eq: ProximityWIthEta}, we find

\[
\begin{split}
\|u^2-&u^1\|_{C^{(2+\delta)/2,2+\delta}}+\frac{\|(m^2(t)-m^1(t))-(m^2(s)-m^1(s))\|_{(C^{2+\delta})'}}{|t-s|^{\delta/2}}\\
&\leq C\Big(\bd_1(m_0^2,m_0^1)+\sup\limits_{t\in [t_0,T]}\bd_1^2(m^2(t),m^1(t))\Big)
\end{split}\]
\[\leq C\Big(\bd_1(m_0^2,m_0^1)+\|u^2-u^1\|_{C^{(2+\delta)/2,2+\delta}}^2 +\bd_1^2(m_0^2,m_0^1)\Big),\]

where in the last inequality we used \eqref{eq: D1Sup}. 
\vs
Therefore, choosing $\eta>0$ even smaller if necessary,  we obtain 
\[\|u^2-u^1\|_{C^{(2+\delta)/2,k+\delta}}+\frac{\|(m^2(t)-m^1(t))-(m^2(s)-m^1(s))\|_{(C^{2+\delta})'}}{|t-s|^{\delta/2}}\leq C\bd_1(m_0^2,m_0^1).\]
Using the last inequality in \eqref{eq: D1Sup} yields
\[\sup\limits_{t\in [t_0,T]}\bd_1(m^2(t),m^1(t))\leq C\bd_1(m_0^2,m_0^1).\]

\end{proof}

Next, we give a sharpened version of the main regularity result in \cite{CardSoug2022}.
\begin{prop} \label{prop.dmlip}
    Let Assumption \ref{assump:values} hold. Then, for each $(t_0,m_0) \in \cO$, there exist constants $\delta, C > 0$ such that, for  $t$, $m_1,m_2$ with $|t - t_0| < \delta$, $\dtwo(m_0,m_i) < \delta$ and $i = 1,2$, we have 
    \begin{align*}
        \sup_{x \in \R^d} |D_m \cU(t,m_1,x) - D_m \cU(t,m_2,x)| \leq C \bd_1(m_1,m_2).
    \end{align*}
\end{prop}
\begin{proof}
For $i=1,2$, let $(u^i,m^i)$ be solutions to \ref{system: MFGSystem} with initial conditions $m^i(t)=m_i$. We have that
    \[D_m \cU(t,m^i,x)= Du^i(t,x)\]
    and the result follows from Lemma \ref{lem: MainEstimatesFOrDifferenceOfSOlutions}.

\end{proof}

\subsection{The $\cC^2-$regularity of $\cU$}
In this subsection we show  that the function $\cU$ is twice differentiable in $m$ and, moreover, the map $m\rightarrow D_{mm}\cU(t,m,x,y)$ is Lipschitz in $d_1$ locally within $\cO$. This is the assertion of Theorem~\ref{thm.c2}. The proof  follows closely the one developed \cite{CardaliaguetDelarueLasryLions}. 
\vs
We  introduce next some notation, and also give a roadmap showing how the arguments of \cite{CardaliaguetDelarueLasryLions} will be adapted to the present setting.
\vs
For $(t_0,m_0) \in \cO$, let $m^{(t_0,m_0)}$ denote the unique optimal trajectory started from $(t_0,m_0)$ and by $u^{(t_0,m_0)}$ its  corresponding multiplier, and consider the map $\Phi : \cO \times \R^d \to \R$ given  by 
\[\Phi(t_0,m_0,x)=u^{(t_0,m_0)}(t_0,x).\]
It was shown in the proof of Lemma 2.9 in \cite{CardSoug2022} that, for $(t,m) \in \cO$, we have
\begin{align} \label{eq: FormulaForDelta1U}
    \frac{\delta \cU}{\delta m}(t,m,x) = \Phi(t,m,x) - \int_{\R^d} \Phi(t,m,z) m(dz).
\end{align}
Given a multi-index $l\in \{0,1\}^d$ with $|l|=\sum\limits_{i=1}^dl_i\leq 1$, that is,  either $l = (0,\cdots,0)$ or $l =e_i$ for some $i\in \{1,\cdots,d\}$, where $e_i$ is the standard basis in $\R^d$, and $y\in \R^d$, let $(w^{(l),y},\rho^{(l),y})$ be the solution to \eqref{system: LinearizedSystemNoError} driven by $(u,m)$, with initial condition $\rho^{(l),y}(t_0,\cdot)=(-1)^{|l|}D^{(l)} \delta_y$. 
\vs
We also define the function $K^{(l)} : \cO \times \R^d \times \R^d \to \R$ given by
\[K^{(l)}(t_0,m_0,x,y)=w^{(l),y}(t_0,x),\]
and, for simplicity of notation, we write $K = K^{(0)}$. 
\vs
Following the arguments in \cite{CardaliaguetDelarueLasryLions}, we show  below that 
\begin{align}\label{eq: LinearDerivativeOfPhi}
    \frac{\delta \Phi}{\delta m}(t,m,x,y) = \frac{\delta}{\delta m} \big[\Phi(t,\cdot, x) \big](y) = K(t_0,m_0,x,y).
\end{align}
We note that the normalization convention \[\intd K(t_0,m_0,x,y)m_0(dy) =0\]
is satisfied, since $(z,\rho) = (0,m^{(t_0,m_0)})$ is the unique solution to \eqref{system: LinearizedSystemNoError}.
\vs
Combining  \eqref{eq: LinearDerivativeOfPhi} and  \eqref{eq: FormulaForDelta1U} and keeping in mind the normalization convention for linear derivatives, it follows that $\frac{\delta^2 \cU}{\delta m^2} = \cK$, where $\cK$ is the ``normalized" version of $K$ given by 
\begin{equation}\label{eq: DefinitionOfNormalizedK}
\begin{split}
\cK(t_0,m_0,x,y)&= K(t_0,m_0,x,y)-\intd K(t_0,m_0,z,y)m_0(dz)-u^{(t_0,m_0)}(t_0,y)\\
&+\intd u^{(t_0,m_0)}(t_0,z)m_0(dz).
\end{split}
\end{equation}

The existence and regularity of $D_{mm} \cU$ thus follows from the regularity properties of the map $K$, which we  investigate next.
\begin{prop} \label{prop.kreg}
Let Assumption \ref{assump:gradients} hold and fix $(t_0,m_0)\in \mathcal{O}$. Then, 
the function $K^{(0)}(t_0,m_0,x,y)$ is differentiable in $y$. Moreover, for any $l\in \{0,1\}^d$ with $|l|\leq 1$, the derivative $x\rightarrow D_y^{(l)}K^{(0)}(t_0,m_0,x,y)$ belongs to $C^{2+\delta}(\R^d)$ and is given by
\begin{equation}\label{eq: DifferentiabilityOfSecondDerivative}
    D_y^{(l)}K^{(0)}(t_0,m_0,x,y)= K^{(l)}(t_0,m_0,x,y).
\end{equation}
Furthermore, there exists a constant $C>0$, which depends on the data, such that
\begin{equation}\label{eq: GrowthOfSecondDerivative}
    \sup\limits_{y\in \R^d}\| D_y^{(l)}K^{(0)}(t_0,m_0,\cdot,y)\|_{C^{2+\delta}}\leq C,
\end{equation}
and
\begin{equation}\label{eq: HolderRegularityOfSecondDerivative}
    \|D_y^{(l)}K^{(0)}(t_0,m_0,\cdot,y')-D_y^{(l)}K^{(0)}(t_0,m_0,\cdot,y)\|_{2+\delta}\leq C|y'-y|^{\delta}.
\end{equation}
Finally, given a finite signed measure $\xi$ on $\R^d$, the unique solution $(z,\rho)$ to \eqref{system: LinearizedSystemNoError} driven by $(u,m)$ and with initial condition $\rho(t_0)=\xi$ satisfies 
\begin{equation}\label{eq: RepresentationSecondDerivative}
    z(t_0,x)=\langle K^{(0)}(t_0,m_0,x,\cdot),\xi\rangle .
\end{equation}
\end{prop}
\begin{proof}
    Fix $y\in \R^d$. Since $|l|\leq 1$, Lemma \ref{lem: MainEstimatesLinearizedSystem} yields 
    \begin{align*}
        &\|w^{(l),y}\|_{C^{(2+\delta)/2,2+\delta}}+\sup\limits_{t\in [t_0,T]}\|\rho^{(l),y}(t)\|_{(C^{2+\delta})'}+\sup\limits_{t'\neq t}\frac{\|\rho^{(l),y}(t)-\rho^{(l),y}(s)\|_{(C^{2+\delta})'}}{|t-s|^{\delta/2}} \\
        &\qquad \leq C\|D^{(l)}\delta_y\|_{(C^{1+\delta})'}\leq C.
    \end{align*}
   Note that in the estimate  above we used that 
    \begin{align*}
        \|D^{(l)}\delta_{y}\|_{(C^{1+\delta})'}=\sup\limits_{\|f\|_{C^{1+\delta}}\leq 1}D_y^{(l)}f(y)\leq \sup\limits_{\|f\|_{C^{1+\delta}}\leq 1}\|f\|_{C^{1+\delta}}\leq 1.
    \end{align*}
   Let  $e_1, \ldots, e_d$ be  the standard basis vectors in $\R^d$. Then, for $\epsilon>0$ and $i\in \{1,\cdots,d\}$,  the pair
    \[(z_i^{\epsilon},\lambda_i^{\epsilon})=\Bigg(\frac{1}{\epsilon}\Big(w^{(0),y+\epsilon e_i}-w^{(0),y}\Big) -w^{(e_i),y},\frac{1}{\epsilon}\Big(\rho^{(0),y+\epsilon e_i}-\rho^{(0),y}\Big) -\rho^{(e_i),y}\Bigg)\]
    is the unique solution to \eqref{system: LinearizedSystemNoError} driven by $(u,m)$ with initial condition
    \[\lambda_i^{\epsilon}=\frac{1}{\epsilon}\Big(\delta_{y+\epsilon e_i}-\delta_y\Big)-(-1)D^{(e_i)} \delta_y.\] 
  Since  
    \begin{align*}
        &\|\frac{1}{\epsilon}\Big(\delta_{y+\epsilon e_i}-\delta_y\Big)-(-1)D^{(e_i)} \delta_y\|_{ (C^{1+\delta})'}\\
        &\quad =\sup\limits_{\|f\|_{C^{1+\delta}}\leq 1} \frac{f(y+\epsilon e_i)-f(y)-\epsilon D^{(e_i)}f(y)}{\epsilon}\\
        &\quad =\sup\limits_{\|f\|_{C^{1+\delta}}\leq 1} \int_0^1 D^{(e_i)}\Big(f(s(y+\epsilon e_i)+(1-s)y)-f(y)\Big) ds\\
        &\quad \leq C\sup\limits_{\|f\|_{C^{1+\delta}}\leq 1} \|D^{e_i}f\|_{C^{\delta}}\int_0^1 s^{\delta}\epsilon^{\delta}ds\leq C\sup\limits_{\|f\|_{C^{1+\delta}}\leq 1}\|f\|_{C^{1+\delta}}\epsilon^{\delta}\leq C\epsilon^{\delta }, 
        \end{align*}
 Lemma \ref{lem: MainEstimatesLinearizedSystem} yields 
    \begin{align*}
        &\|\frac{1}{\epsilon}\Big(w^{(0),y+\epsilon e_i}-w^{(0),y}\Big) -w^{(e_i),y}\|_{C^{(2+\delta)/2,2+\delta}} \\
        &\qquad \leq C\|\frac{1}{\epsilon}\Big(\delta_{y+\epsilon e_i}-\delta_y\Big)-(-1)D^{(e_i)} \delta_y\|_{(C^{1+\delta})'}\leq C\epsilon^{\delta},
    \end{align*}
and, hence,    \eqref{eq: GrowthOfSecondDerivative} and \eqref{eq: DifferentiabilityOfSecondDerivative} follow. 
\vs 

Furthermore, given $y',y$,  \eqref{eq: HolderRegularityOfSecondDerivative} is an application of Lemma \ref{lem: MainEstimatesLinearizedSystem} to the pair $(w^{(l),y'}-w^{{(l),y}},\rho^{(l),y'}-\rho^{(l),y})$. 
\vs
    Finally,  \eqref{eq: RepresentationSecondDerivative} follows from the fact  that the pair 
    \[(z^{\xi},\mu^{\xi})=( \intd w^{(0),y}(t,x)d\xi(y), \intd \rho^{(0),y} d\xi(y))\]
    is the unique solution to \eqref{system: LinearizedSystemNoError} driven by $(u,m)$ with initial condition $\rho^{\xi}(t_0)=\xi$. 

\end{proof}
We now show the Lipschitz continuity of $K^{(l)}$ with respect to $m$.
\begin{prop}\label{prop: C^2,1}
Let Assumption \ref{assump:gradients} hold. Given $(t_0,m_0^1)\in \mathcal{O}$, there exist $\eta >0$ and $C>0$ such that, if $m_0^2\in \mathcal{P}_2(\R^d)$ and $\bd_2(m_0^2,m_0^1)\leq \eta$, then
\[\|K^{(l)}(t_0,m_0^2,\cdot,y)-K^{(l)}(t_0,m_0^1,\cdot,y)\|_{2+\delta}\leq C\bd_1(m_0^2,m_0^1).\]
\end{prop}
\begin{proof}
    Let $(u^1,m^1)$ and $ (u^2,m^2)$ be the unique stable solutions to \eqref{system: MFGSystem} with initial conditions $m^1(t_0)=m_0^1$ and $m^2(t_0)=m_0^2$ respectively. In addition, let $(z^1,\rho^1), (z^2,\rho^2)$ be solutions to \eqref{system: LinearizedSystemNoError} driven by $(u^1,m^1),(u^2,m^2)$ respectively and with initial conditions $\rho^1(t_0)=\rho^2(t_0)= (-1)^{|l|}D^{(l)}\delta_y$. 
 \vs   
    The pair 
    \[(w,\lambda) =(z^2-z^1,\rho^2-\rho^1)\]
    solves \eqref{system: LinearizedSystemWithError} driven by $(u^1,m^1)$ with 
    \begin{align*}
        &R^1 =\Big(D_pH(x,Du^1)-D_pH(x,Du^2) \Big)\cdot Du^2+\Big(\frac{\delta F}{\delta m}(x,m^2(t))(\rho^2(t))-\frac{\delta F}{\delta m}(x,m^1(t))(\rho^2(t)) \Big), \\
        &R^2 =\rho^2(t)\Big(D_pH(x,Du^2)-D_pH(x,Du^1)\Big)+\Big(m^2D_{pp}H(x,Du^2)-m^1D_{pp}H(x,Du^1)\Big)\cdot Du^2, \\
        &R^3 =\frac{\delta G}{\delta m}(x,m^2(T))(m^2(T))-\frac{\delta G}{\delta m}(x,m^1(T))(m^2(T)), \\
        &\xi =0.
    \end{align*}
From Lemma \ref{lem: MainEstimatesFOrDifferenceOfSOlutions} we have 

\begin{align*}
    \|R^3\|_{C^{2+\delta}} \leq C\sup\limits_{t\in [t_0,T]}\bd_1(m^2(t),m^1(t))\leq C\bd_1(m_0^2,m_0^1), 
    \end{align*}
    and
    \begin{align*}
    &\sup\limits_{t\in [t_0,T]}\|R^2(t)\|_{(W^{1,\infty})'} \leq C\Big(\sup\limits_{t\in [t_0,T]}\|Du^2(t)-Du^1(t)\|_{C^{1}}+\sup\limits_{t\in [t_0,T]}\bd_1(m^2(t),m^1(t))\Big)\\[1mm]
    & \hskip1.5in \leq C\bd_1(m_0^2,m_0^1).
\end{align*}
\vs
It remains to estimate $\|R^1\|_{C^{\delta/2,\delta}}$. To this end,  we rewrite it as
\begin{align*}
R^1 &=A+B,
\end{align*}
with 
\[
    A=\Big(D_pH(x,Du^1)-D_pH(x,Du^2) \Big)\cdot Du^2,\]
    and 
   \[ B=\frac{\delta F}{\delta m}(x,m^2(t))(\rho^2(t))-\frac{\delta F}{\delta m}(x,m^1(t))(\rho^2(t)).
   \]
It follows from Lemma \ref{lem: MainEstimatesFOrDifferenceOfSOlutions} that
\[\|A\|_{C^{\delta/2,\delta}}\leq C\|D(u^2-u^1)\|_{C^{\delta/2,\delta}}\leq C\|u^2-u^1\|_{C^{(2+\delta),2+\delta}}\leq C\bd_1(m_0^2,m_0^1).\]
Finally, we write
\[B(t,x)=\int_0^1 \frac{\delta^2 F}{\delta m^2}(x,\lambda m^2(t)+(1-\lambda)m^1(t))(\rho^2(t))(m^2(t)-m^1(t))d\lambda. \]
An argument similar to the one in  the proof of Lemma \ref{lem: MainEstimatesFOrDifferenceOfSOlutions} yields 
 \[\|B(\cdot,\cdot)\|_{C^{\delta/2,\delta}}\leq C\Big(\sup\limits_{t\neq s}\frac{\|(m^2(t)-m^1(t)-(m^2(s)-m^1(s)))\|_{(C^{2+\delta})'}}{|t-s|^{\delta/2}}+\sup\limits_{t\in [t_0,T]}\bd_1(m^2(t),m^1(t))\Big),\]
and, hence, by Lemma \ref{lem: MainEstimatesFOrDifferenceOfSOlutions},  we have 
\[\|R^1\|_{C^{\delta/2,\delta}}\leq C\bd_1(m_0^2,m_0^1).\]
Finally,  Lemma \ref{lem: MainEstimatesLinearizedSystem} and the definition of $K^{(l)}$ imply that 
\[\|K^{(l)}(t_0,m_0^2,\cdot,y)-K^{(l)}(t_0,m_0,\cdot,y)\|_{2+\delta}\leq \|z^2-z\|_{C^{(2+\delta)/2,2+\delta}}\leq C\bd_1(m_0^2,m_0^1).\]
\end{proof}
\begin{thm}\label{thm: C^2Theorem}
Let Assumption \ref{assump:gradients} hold. The map $\mathcal{U}$ is $\cC^2$ in the set $\mathcal{O}$ and satisfies
\[\frac{\delta^2 \cU}{\delta^2 m}(t_0,m_0,x,y) = \cK(t_0,m_0,x,y).\]
Moreover, given $(t_0,m_0^1)\in \mathcal{O}$, there exist an $\eta>0$ and $C>0$ such that, if $m_0^2\in \mathcal{P}_2(\R^d)$ with $d_2(m_0^2,m_0^1)\leq \eta$, then
\[\Big\|\frac{\delta \cU}{\delta m}(t_0,m_0^2,\cdot)- \frac{\delta \cU}{\delta m}(t_0,m_0^1,\cdot)-\intd \cK(t_0,m_0,\cdot,y)d(m_0^2-m_0^1)(y)\Big\|_{2+\delta}\leq C\bd_1^2(m_0^2,m_0^1).\]
\end{thm}
\begin{proof}
    Let $(u^1,m^1)$ and $(u^2,m^2)$ be the unique stable solutions to \eqref{system: MFGSystem} with initial conditions $m^1(t_0)=m_0^1,m^2(t_0)=m_0^2$ respectively, and $(z,\mu)$ be the solution to \eqref{system: LinearizedSystemNoError} driven by $(u,m)$ with initial condition $\rho(t_0)=m_0^2-m_0^1$. 
    
    \vs
    The pair
    \[(z,\rho)=(u^2-u^1-z,m^2-m^1-\mu)\]
    solves the linearized system \eqref{system: LinearizedSystemWithError} driven by $(u^1,m^1)$, with $\xi=0$ and 
    \begin{align*}
        R^1 &=
            -(H(x,Du^2)-H(x,Du^1)-D_pH(x,Du^1)\cdot D(u^2-u^1))\\
            &\qquad +F(x,m^2)-F(x,m^1)-\frac{\delta F}{\delta m}(x,m^1(t))(m^2(t)-m^1(t)), \\
        R^2 &= 
        (D_pH(x,Du^2)-D_pH(x,Du^1)-D_{pp}H(x,Du^1)\cdot(Du^2-Du^1))m^1\\
        &\qquad +(D_pH(x,Du^2)-D_pH(x,Du^1))(m^2-m^1), \\
        R^3 &= G(x,m^2(T))-G(x,m^1(T))-\frac{\delta G}{\delta m}(x,m^1(T))(m^2(T)-m^1(T)).
    \end{align*}
    Arguments similar to those used to bound $R^1$ in the proof of \ref{lem: MainEstimatesFOrDifferenceOfSOlutions} yield
    
    \begin{align*}
        &\|R^1\|_{C^{\delta/2,\delta}}+\sup\limits_{t\in [t_0,T]}\|R^2\|_{(W^{1,\infty})'}+\|R^3\|_{(C^{2+\delta})'} \\
        &\quad \leq C\Big(\|Du^2-Du^1\|_{C^{\delta/2,\delta}}^2+\sup\limits_{t\in [t_0,T]}\bd_1^2(m^2(t),m^1(t))+\Big(\sup\limits_{t\neq s}\frac{\|(m^2(t)-m^1(t))-(m^2(s)-m^1(s))\|}{|t-s|^{\delta/2}}\Big)^2 \Big),
    \end{align*}
    and thus Lemma~\ref{lem: MainEstimatesLinearizedSystem} and Lemma~\ref{lem: MainEstimatesFOrDifferenceOfSOlutions} imply that  
    \[\|u^2-u^1-z\|_{C^{(2+\delta)/2,2+\delta}}\leq C\bd_1^2(m_0^2,m_0^1).\]
    The above shows that
    \[\frac{\delta \Phi}{\delta m}(t_0,m_0^1,x,y) = K(t_0,m_0,x,y).\]
    The result now follows from the representations \eqref{eq: DefinitionOfNormalizedK}, \eqref{eq: FormulaForDelta1U} and \eqref{eq: RepresentationSecondDerivative}.

\end{proof}

We may now show Theorem \ref{thm.c2}.
\begin{proof}[The proof of Theorem \ref{thm.c2}]
    The claim  follows from Proposition \ref{prop: C^2,1} and Theorem \ref{thm: C^2Theorem}.

\end{proof}
We conclude this section by showing that the function $D_m\mathcal U$ is Lipschitz continuous with respect to time.
\begin{prop} \label{prop.timereg}
    Let Assumption \ref{assump:gradients} hold. Given $(t_0,m_0)\in \mathcal{O}$, there exist $C>0$ and $\delta>0$ depending on the data such that, for all $|h|<\delta $,
    \begin{equation*}
        \Big\|D_m\cU ((t_0+h)\wedge T,m_0,\cdot)-D_m\cU(t_0,m_0,\cdot)\Big\|_{\linf} \leq C|h|.
    \end{equation*}
\end{prop}
\begin{proof}
Fix $(t_0,m_0) \in \cO$ and $x \in \R^d$.
\vs

We write 
\begin{align*}
    |D_m \cU(t_0 & + h,m_0 x) - D_m \cU(t_0, m_0,x)| \leq I + II, 
    \end{align*}
    with 
    \begin{align*}
    I = |D_m \cU(t_0 + h, m_{t_0}, x) - D_m \cU(t_0 +h, m_{t_0 + h}, x)|, 
    \end{align*}
    and 
    \begin{align*}
    II  = |D_m \cU(t_0 +h, m_{t_0 + h}, x) - D_m \cU(t_0 + h, m_{t_0}, x)|,
\end{align*}
where $t \mapsto m_t$ is the unique optimal trajectory started from $(t_0,m_0)$, and proceed obtaining bounds for $I$ and $II$.
\vs
To estimate $II$, we note that the regularity of $u^{(t_0,m_0)}$ yields 
\begin{align*}
    II = |D_x u^{(t_0,m_0)}(t_0 + h,x) - D_x u^{(t_0,m_0)}(t_0,x)| \leq Ch. 
\end{align*}
For the term $I$, we first note that,  in view of  Theorem \ref{thm: C^2Theorem} and the regularity of $K$ proved in Proposition \ref{prop.kreg}, we can find $\delta$ small enough that, if $(t,m)$ is such that  $|t - t_0| < \delta$ and $\bd_2(m,m_0) < \delta$, then, for some $C > 0$, 
\begin{align*}
   \| D_{mm} \cU(t,m,\cdot,y) \|_{C^{1 + \delta}} \leq C. 
\end{align*}
Since $D_{mm} \cU(t,m,x,y) = D_{mm} \cU(t,m,y,x)^T$ by Corollary 5.89 in \cite{CarmonaDelarue_book_I}, it follows that $$\|D_{mm} \cU(t,m,x,\cdot)\|_{C^{1 + \delta}} \leq C.$$
From here it is straightforward to check that there is a constant $C$ such that,  if  $|t - t_0| < \delta$ and $\bd_2(m,m_0) < \delta$, we have 
\begin{align*}
    \| \frac{\delta }{\delta m} \big[ D_m\cU(t,m,x) \big](\cdot) \|_{C^2} \leq C.
\end{align*}
It follows that $D_m \cU$ is locally Lipschitz in  $(C^2)'$, and, in particular, for $|t - t_0| < \delta$, $\bd_2(m,m_0) < \delta$, we find  
\begin{align*}
    |D_m \cU(t,m,x) - D_m \cU(t,m_0,x)| \leq C \|m - m_0\|_{(C^2)'}.
\end{align*}
Note also that standard estimates yield that  $t \mapsto m_t$ is Lipschitz with respect to the $(C^2)'$-metric. 
Thus, choosing, if necessary, $\delta$ even smaller  so that $\bd_2(m_{t_0 + h}, m_{t_0})$ is small enough, we find that 
\begin{align*}
    I \leq C \|m_{t_0 + h} - m_{t_0}\|_{(C^2)'} \leq Ch,
\end{align*}
which completes the proof. 
\end{proof}

\bibliographystyle{plain}
\bibliography{CJMSreferences}

\end{document}